\documentclass[11pt]{amsart}
\usepackage{amscd,amssymb}
\usepackage[all]{xy}
\usepackage{stmaryrd}
\title{Non-adic formal schemes}
\author{Takehiko Yasuda}
\thanks{Research Fellow of the Japan Society for the Promotion of Science}
\address{Research Institute for Mathematical Sciences, Kyoto University, 
Kyoto, 606-8502, Japan, Fax:075-712-1836}

\email{takehiko@kurims.kyoto-u.ac.jp}


\theoremstyle{plain}
\newtheorem{thm}{Theorem}[section]

\newtheorem{prop}[thm]{Proposition}
\newtheorem{cor}[thm]{Corollary}
\newtheorem{lem}[thm]{Lemma}
\newtheorem{prop-defn}[thm]{Proposition-Definition}
\newtheorem{cor-defn}[thm]{Corollary-Definition}

\theoremstyle{definition}
\newtheorem{defn}[thm]{Definition}
\newtheorem{expl}[thm]{Example}

\newtheorem{conv}[thm]{Convention}

\theoremstyle{remark}
\newtheorem{rem}[thm]{Remark}

\def\AA{\mathbb A}
\newcommand{\CC}{\mathbb C}

\newcommand{\ZZ}{\mathbb Z}

\newcommand{\PP}{\mathbb P}

\newcommand{\NN}{\mathbb N}
\newcommand{\No}{\mathbb{N}_0}

\newcommand{\cF}{\mathcal{F}}
\newcommand{\cG}{\mathcal{G}}
\newcommand{\cH}{\mathcal{H}}
\newcommand{\cI}{\mathcal{I}}
\newcommand{\cJ}{\mathcal{J}}
\newcommand{\cK}{\mathcal{K}}
\newcommand{\cL}{\mathcal{L}}
\newcommand{\cM}{\mathcal{M}}
\newcommand{\cN}{\mathcal{N}}
\newcommand{\cO}{\mathcal{O}}

\newcommand{\cQ}{\mathcal{Q}}

\newcommand{\bA}{\mathbf{A}}

\newcommand{\bC}{\mathbf{C}}
\newcommand{\bD}{\mathbf{D}}

\newcommand{\bF}{\mathbf{F}}

\newcommand{\bK}{\mathbf{K}}
\newcommand{\bP}{\mathbf{P}}

\newcommand{\cX}{\mathcal{X}}

\newcommand{\fp}{\mathfrak{p}}

\newcommand{\fm}{\mathfrak{m}}

\newcommand{\fB}{\mathfrak{B}}

\newcommand{\Ker}{\mathrm{Ker}\,}
\newcommand{\Coker}{\mathrm{Coker}\,}

\newcommand{\Hom}{\mathrm{Hom}}
\newcommand{\Spec}{\mathrm{Spec}\,}
\newcommand{\Spf}{\mathrm{Spf}\,}

\def\Im{\mathrm{Im}}

\newcommand{\Supp}{\mathrm{Supp}\,}

\newcommand{\red}{\mathrm{red}}

\newcommand{\adic}{\mathrm{adic}}

\newcommand{\pro}{\mathbf{pro}\text{-}}

\def\prolim{``\varprojlim"}
\def\et{\textup{\'et}}
\def\Zar{\mathrm{Zar}}
\def\sur{\mathrm{epi}}

\numberwithin{equation}{section}

\begin{document}

\maketitle

\begin{abstract}
Our purpose is to make a contribution to the foundation of the theory
of formal scheme. We are interested particularly in non-Noetherian or 
non-adic formal schemes, which have been little studied.
We redefine the formal scheme as a proringed space and
study its basic properties.
We also find several examples of non-adic formal schemes.
\end{abstract}

\tableofcontents

\section*{Introduction}

The formal scheme is an important tool
for the infinitesimal analysis in the algebraic geometry.
In the original reference \cite{EGA},
Grothendieck defined the formal scheme, which is not
a priori locally Noetherian or adic.
But he made some arguments only under the 
locally Noetherian assumption.
Also in most literature, one considers only locally Noetherian
or at least adic formal schemes.
 The aim of this paper is to complement 
the theory of not necessarily adic formal schemes.
As far as I know,  
the first attempt in this direction after \cite{EGA} is
McQuillan's one \cite{McQuillan}.

Our first task is to redefine the formal scheme as a  proringed space (\S \ref{sec-formal-sch-pro}). 
(Recall that Grothendieck  defined a formal scheme as a topologically ringed space.)
To do this, we make full use of the  procategroy  (see \S \ref{sec-pro}). 
Roughly speaking, the procategory of a category $\bC$
consists of projective systems in $\bC$ with appropriate Hom-sets.
Now, by definition, a   proring is an object of the procategory of the category of rings
and a proringed space is a topological space with a sheaf of prorings.
For a proring $A$ satisfying some condition, which we call 
an admissible proring, 
we define the  formal spectrum,
$\Spf A$, and gluing formal spectrums, we obtain a  formal scheme.
In fact, a formal scheme is not only a proringed space, but also a
locally admissibly proringed space. 
Then a morphism of formal scheme is defined to be a morphism of 
locally admissibly proringed spaces. 
We will see in \S \ref{sec-functors} that the category of formal schemes,
as well as that of schemes,
embeds into  to the category of contravariant functors 
\[
\text{(Schemes)} \to \text{(Sets)}. 
\]
Then we will define a formal algebraic space as a functor.

We will introduce the notion of  semicoherent promodules on a formal scheme (\S \ref{sec-semicoh}),
which play a role of quasicoherent modules on a scheme.
Following McQuillan's observation \cite{McQuillan},
we define the semicoherent promodule, locally on an affine open subset $\Spf A$, 
as the sheafification $M^{\triangle}$ of an $A$-promodule $M$,  
and do not impose the existence of local free
presentation like
\[
\cO_{X}^{J} \to \cO_{X}^{I} \to \cM \to 0 .
\]
Indeed, if a semicoherent promodule $\cM$ is not Mittag-Leffler (see \S \ref{subsec-ML-sheaf}),
then there is not even a local epimorphism $\cO_{X}^{I} \twoheadrightarrow \cM$.
The semicoherent promodule
is much like the quasicoherent module on a scheme.
For a technical reason, for a sheaf $\cM$ on a formal scheme $X$,
we only consider the values $\cM (U)$ for quasi-compact and quasi-separated open 
$U \subseteq X$. It enables us to avoid some troublesome {\em infinite} projective limits.

The local study of formal schemes reduces 
to the study of admissible prorings.
However the proring does not seem suitable to discuss specific examples
and  the complete (linearly topologized) ring seems better for this purpose.
The category of complete rings embeds naturally and fully faithfully into
that of prorings.
We define an  admissible ring as a complete ring 
whose associated proring is admissible.
 It is equivalent to 
McQuillan's admissible ring and more general than
Grothendieck's one.

Similarly, the category of complete modules over a complete ring
embeds into that of promodules over the corresponding proring.
The former category is additive but not abelian\footnote{Correspondingly the category of quasicoherent modules in \cite{McQuillan}
is not abelian  and Claim 5.3 in op.\ cit.\ is not correct. This observation is due to
 a referee.}, while the latter is abelian.
(It is the greatest advantage of the use of promodule.)
We say that  a  proring or promodule is mild
if it  is isomorphic to the one associated to a complete one.
The mildness is a key when moving from the ``pro side'' to the ``complete side''. 
It is worth noticing that if a proring or promodule satisfies the Mittag-Leffler condition
(for instance, every admissible proring does)
and is indexed by a countable set, then it is mild.
So in practice, most admissible prorings are mild
and considering a mild admissible proring is 
equivalent to considering an admissible ring.
The central purpose of the use of prorings is to ensure consistency
with promodules, the use of which is more essential.
(See \S \ref{sec-mild-gentle} for complete rings and modules, and the mildness and so on.)
We will study properties of admissible rings like adic, Noetherian
and so on, and relations between them and some classes
of formal subschemes (\S \ref{sec-local-properties}). 

At the end of the paper (\S \ref{sec-foliation}), 
we construct  non-adic formal schemes from singularities of
foliations. It first motivated me to begin this research.
Arising naturally from the geometry, 
it looks more interesting than the other examples of non-adic formal schemes
in the paper.
A key of the construction is Jouanolou's theorem on the nonexistence
of formal separatrix.

\begin{conv}
We denote by $\NN$ the set of positive integers, 
and by $\No$ the set of non-negative integers.
A {\em ring} means a commutative ring with unit.
A {\em projective (resp.\ inductive) system} means a projective 
(resp.\ inductive) system indexed by
a preordered set, which is also called a generalized projective 
(resp.\ inductive) system
in literature. 
A {\em directed projective (resp.\ inductive) system} means a projective (resp.\ inductive) system
indexed by a directed set.
We write a projective or inductive system as $(X_{d})_{d\in D}$
or simply $(X_{d})$. 
The morphisms in  a projective or inductive system are called {\em
bonding morphisms}.
We denote the category of sets by $\text{(Sets)}$, that of rings by
$\text{(Rings)}$ and so on.
For categories $\bC$ and $\bD$
which admits finite projective and inductive limits,
a (covariant) functor $\bC \to \bD$ is said to be 
{\em left  exact (resp.\ right exact, exact)}
 if it commutes with finite projective  limits
 (resp.\ finite inductive limits, finite projective and inductive limits).
\end{conv}

\subsection*{Acknowledgment}
I would like to thank Fumiharu Kato for useful discussions.
Also I gratefully acknowledges the many helpful comments of two referees.
Especially suggestions of one of them helped me a lot to pursue the generality.

\section{Prorings and promodules}\label{sec-pro}

In this section, we review some of the standard facts on 
procategories, prorings and promodules, which are required in subsequent 
sections.

\subsection{Procategories}

\begin{defn}
Let $\bC$ be a category.
A {\em pro-object} of $\bC$ is a directed projective system in $\bC$.
The  {\em procategory} of $\bC$, denoted $\pro \bC$, is defined as follows:
An object of $\pro \bC$ is a pro-object of $\bC$.
For objects $X= (X_{d})$ and 
$Y= (Y_{e})$ of $\pro\bC$, the set of morphisms is defined by
\[
\Hom _{\pro\bC} (X,Y) := \varprojlim_{e} \varinjlim _{d} \Hom_{\bC} (X_{d},Y_{e}).
\]
\end{defn}

Note that the index sets are not supposed to be equal.

The original references on the procategory are
 \cite{SGA4,AM}. In these references, 
 projective systems are more generally labeled by filtering categories.
But it is proved in \cite{MS} that this  leads
to an equivalent category. 

An object $X \in \bC$ is considered as 
a projective system indexed by a singleton,
and as an object of $\pro \bC$.
This makes $\bC$ a full subcategory of $\pro\bC$.
It is tautology that
for a directed projective system $(X_{d})$ in $\bC$,
its projective limit in $\pro \bC$ 
is $(X_{d})$ itself as an object of $\pro\bC$.
Even if the projective limit in $\bC$ of $(X_{d})$ exists,
it is not generally isomorphic to the one in $\pro\bC$.
In general, we denote the projective limit in $\bC$ by $\varprojlim$,
and the one in $\pro\bC$ by $\prolim $.

Let $A=(A_{d})_{d \in D} \in \pro \bC$ and
$\phi:E \to D$  an order-preserving  map of directed sets.
We obtain a new pro-object $A_{\phi} := (A_{\phi(e)})_{e \in E} $.
If $\phi$ is cofinal (that is,  $\forall d \in D$, $\exists e \in E$, $\phi(e) \ge d$),
then $A $ and $A_{\phi}$ are canonically isomorphic.
We say that $A_{\phi}$ is the {\em reindexing} of $A$ 
or that {\em $A_{\phi} $ is obtained by reindexing $A$ by $\phi$}. 

For $A=(A_{d})_{d\in D},B=(B_{e})_{e\in E} \in\pro \bC $, 
a morphism $f:A \to B$ is by definition a collection of 
morphisms $f_{e}:A \to B_{e}$, $e \in E$.
Each $f_{e}$ is an element of $\varinjlim _{d} \Hom_{\bC}(X_{d},Y_{e})$.
If $f^{d}_{e}:X_{d} \to Y_{e}$ is a representative of the class $f_{e}$,
then we say that $f^{d}_{e}$ {\em represents} $f$.
It is equivalent to the diagram in $\pro \bC$ 
\[
\xymatrix{
A \ar[r]^{f} \ar[d] & B  \ar[d] \\
 A_{d} \ar[r]_{f^{d}_{e}} & B_{e}
}
\]
is commutative.

A directed projective system $(f_{d} :A_{d} \to B_{d})$ of 
morphisms in $\bC$ is called a {\em level morphism} 
and induces a morphism $f:(A_{d}) \to (B_{d})$ in $\pro\bC$
in an obvious way.
If every $f_{d}$ is a monomorphism or an epimorphism,
then so is $f$.
If $A \cong (A_{d})$ and $B \cong (B_{d})$ are isomorphisms
in $\pro\bC$ and if $f'$ denote the composite, 
$A \cong (A_{d}) \xrightarrow{f} (B_{d}) \cong B$,
then we say that $(f_{d})$ {\em represents} $f'$.

Conversely,  any morphism $f : A \to B$ in $\pro \bC$ can be
represented by  a level morphism in $\bC$.
Moreover $f$ can be represented by a level morphism
$(f_{d}:A_{d} \to B_{d})$ where $(A_{d})$ and $(B_{d})$ are reindexings of 
$A$ and $B$ respectively.

\begin{prop}\label{prop-procat-properties}
\begin{enumerate}
\item If $\bC$ has finite projective limits, then $\pro\bC$ has 
projective and inductive limits.
\item Suppose $\bC$ has finite projective limits. If  $(A_{d})$ is a projective system
in $\pro\bC$ (each $A_{d}$ is a pro-object of $\bC$)
and $B \in \bC$,
then
\begin{align*}
 \Hom _{\pro\bC} (\prolim A_{d},B)  = \varinjlim \Hom _{\pro\bC} (A_{d},B) . 
 \end{align*}
\item If $\bC$ is additive (resp.\ abelian), then so is $\pro\bC$.
\end{enumerate}
\end{prop}

\begin{proof}
(1) $\pro\bC$ has inductive limits, finite projective limits and directed projective limits
 \cite[Propositions 4.2, 4.3 and 4.4]{AM}. 
 In general, if a category  has finite projective limits and directed projective limits,
 then it has arbitrary projective limits, which proves the assertion.
 Indeed, let $(X_{d})_{d \in D}$ be a projective system labeled by a preordered set $D$ in
 such a category.  
 Let $E$ be the set of the finite subsets of $D$ and for each $A \in E$,
 put $X_{A}:= \varprojlim_{d \in A} X_{d}$. 
 We make $E$ a directed set by $A  \le B$ for $A \subseteq B $, $A,B \in E$.
 For $A \le B$, we have the projection $X_{B} \to X_{A}$ and
 $(X_{A})_{A \in E}$ is a directed projective system.
 By assumption,  $ \varprojlim _{A \in E} X_{A}$ exists. 
 It is now straightforward to check that 
 $\varprojlim _{A \in E} X_{A} = \varprojlim _{d \in D} X_{d}$. 
 
 (2) For every $X =(X_{d}) \in \pro\bC$, define the functor 
 \begin{align*}
 F_{X}: \bC &\to (\text{Sets}) \\ 
  Y &\mapsto \Hom _{\pro\bC}(X,Y) = 
 \varinjlim \Hom_{\bC} ( X_{d},Y).
 \end{align*}
It determines a fully faithful embedding 
$\pro\bC \hookrightarrow \mathrm{Fun} (\bC,(\text{Sets}))^{\circ}$.
For $(A_{d})$ in the assertion, 
define $G \in \mathrm{Fun} (\bC,(\text{Sets}))$  by
\[
 G(Y) := \varinjlim \Hom_{\pro\bC} (A_{d} , Y).
\]
Then $G$ is the projective limit of $(F_{A_{d}})$ in $\mathrm{Fun} (\bC,(\text{Sets}))$.
Since the directed inductive limit  of sets is
exact, $G$ is left exact. 
From \cite[Cor.\ 2.8]{AM}, $G$ is pro-representable, that is,
for some $C \in \pro\bC$, 
\[
G(Y) = \Hom_{\pro\bC}(C,Y).
\]
Then $C$ is the projective limit of $(A_{d})$, which shows the assertion.
 
 (3) \cite[Prop.\ 4.5]{AM}.
\end{proof}

\begin{prop}\label{prop-incl-exact}
The inclusion functor $\bC \hookrightarrow \pro\bC$ is exact. 
Moreover it commutes also with (not necessarily finite) inductive limits.
\end{prop}

\begin{proof}
Equivalent statements for the ind-category are \cite[I, Prop.\ 8.9.1 and 8.9.5]{SGA4}.
\end{proof}

In particular, for an exact sequence in $\bC$ of the form
\[
 0 \to A \to B \rightrightarrows C,
\]
is exact also in $\pro\bC$. If $\bC$ is abelian,
the same holds for a short exact sequence,
\[
0 \to A \to B \to C \to 0.
\]

\subsection{Prorings and promodules}

\begin{defn}
We define the category of {\em abelian progroups} to be the procategory
of the category of abelian groups. 
\end{defn}

From Proposition \ref{prop-procat-properties}, the category of
abelian progroups is abelian.

\begin{defn}
We define the category of {\em  prorings} to be the procategory
of the category of rings.
\end{defn}

\begin{defn}
Let $A$ be a proring.
A ring $B$ endowed with a morphism $A \to B$ is
called an $A$-{\em algebra}.
We define the category of {\em $A$-proalgebras}
as the procategory of the category of $A$-algebras.
\end{defn}

We see that giving an $A$-proalgebra is equivalent to giving
a proring $ B$ endowed with a morphism $A \to B$.

We now fix a proring $A=(A_{d})_{d \in D}$. 

\begin{defn}
An abelian group $M$ is
called an {\em $A$-module} if it is given a compatible system of
 $A_{d}$-actions, $d \in D_{M}$, for some
  residual subset $D_{M}$ of $D$ (that is, 
$\exists d' \in D$, $\forall d \ge d'$, $d \in D_{M}$).
We define the  {\em category of $A$-modules} by defining the Hom-set as follows: For $A$-modules $M$ and $N$, 
\[
 \Hom_{A} (M,N) := \bigcup_{d \in D_{M} \cap D_{N}} \Hom_{A_{d}} (M,N) \subseteq \Hom _{\text{ab.\ gp.}} (M,N).
\]
\end{defn}

If every bonding map of $A$ is surjective, 
then for every $d \in D_{M} \cap D_{N}$,  we simply have
\[
 \Hom_{A} (M,N) = \Hom_{A_{d}} (M,N) .
\]

For an $A$-module $M$ and for $d \in D$,
we say that $A_{d}$ {\em acts} on $M$ if there exists
an $A_{d}$-action on $M$ which is compatible with
the $A_{d'}$-actions for $d' \in D_{M}$.
Then we can safely add $d$ to  $ D_{M}$.

It is easy to see that the category of $A$-modules is abelian
and that in this category, the notions of injection and monomorphism (resp.\
surjection and epimorphism) coincide. 

\begin{defn}
We define the {\em category of $A$-promodules} as the procategory
of the category of $A$-modules. 
\end{defn}

From Proposition \ref{prop-procat-properties},
the category of $A$-promodules is abelian. 

An $A$-proalgebra $B$ is naturally  regarded as  an
$A$-promodule too. In particular, $A$ itself is an $A$-promodule,
but not an $A$-module unless its index set is a singleton.
If $M$ is a $B$-promodule, then $M$ is regarded as an $A$-promodule too,
 which we denote by $M_{A}$. If $B$ is isomorphic to $A$,
 then the functor $M \mapsto M_{A}$ is an equivalence.

Let $M=(M_{e})_{e \in E}$ be an $A$-promodule 
and set
\[
F:= \{ (d,e) | \text{$A_{d}$ acts on $M_{e}$} \} \subseteq D \times E.
\]
Then reindexing $A$ and $M$ by $F$, we may suppose that
$A$ and $M$ has the same index set $D=E$
and for every $d \in D$, $A_{d}$ acts on $M_{d}$.

\begin{defn}
Let $A=(A_{d})$ be a proring,  and $M=(M_{e})$ and $N=(N_{f})$ $A$-promodules.
Then we define the {\em tensor product} $M \otimes _{A} N$ as the 
$A$-promodule $(M_{e} \otimes _{A_{d}} N_{f})$. Here
$(d,e,f)$ runs over the triplets such that $A_{d}$ acts on $M_{e}$
and $N_{f}$.
\end{defn}

If the index sets of $A$ and $M$ are equal  and
if for every $d$, $A_{d}$ acts on $M_{d}$,
then we have a natural isomorphism
\[
M \otimes _{A} N  \cong (M_{d} \otimes_{A_{d}} N_{f}) , 
\]
because the projective system $(M_{d} \otimes_{A_{d}} N_{f})$
is a cofinal subsystem of $(M_{e} \otimes_{A_{d}} N_{f})$.
If the index set of $N$ is also equal, then  similarly
\[
M \otimes_{A} N =(M_{d} \otimes_{A_{d}} N_{d}).
\] 

Now it is obvious that for each $A$-promodule $M$, the functor 
$M \otimes_{A}-$  is right exact. 

\begin{defn}
An $A$-promodule $M$ is said to be {\em flat}
if $- \otimes_{A}M$ is exact.
\end{defn}

If for each $d$, $M_{d}$ is a flat $A_{d}$-module
and $M=(M_{d})$, then $M $ is flat.

If $A$ is a proring and if $B$ and $C$ are $A$-proalgebras, then
$B \otimes _{A} C$ is also an $A$-proalgebra.
We see that $B \otimes _{A} C$ is the sum  in the category
of $A$-proalgebras and hence the fiber sum in the category of prorings.
 
\subsection{Mittag-Leffler and epi pro-objects}
 
Fix an abelian category $\bA$ such that 
any collection $(X_{i})$ of subobjects of  $X \in \bA$
has the infimum, denoted $\bigcap X_{i} \subseteq X$.

\begin{defn}
A pro-object $(X_{d})_{d \in D}$ of $\bA$ is said to be {\em Mittag-Leffler} 
if for every $d \in D$, there exists $d' \in D$ such that
\[
  \bigcap_{e \ge d} \Im (X_{e} \to X_{d}) = \Im (X_{d'} \to X_{d}).
\]
A pro-object  is said to be {\em epi}
if its every bonding morphism is an epimorphism.
A proring is said to be {\em Mittag-Leffler} (resp.\ {\em epi})
if it is so as an abelian progroup.
\end{defn}

Every epi pro-object is obviously Mittag-Leffler. 

\begin{rem}
We define the Mittag-Leffler and epi properties for a proring
by think of it as an abelian progroup.
\end{rem}

\begin{defn}\label{defn-ass-sur}
For a pro-object $X=(X_{d})$, we define a pro-object
$X^{\sur}=(X^{\sur}_{d})$ by for each $d$, $X^{\sur}_{d} :=\bigcap_{e \ge d} \Im (X_{e} \to X_{d}) $.
\end{defn}

The inclusions $X^{\sur}_{d} \hookrightarrow X_{d}$ induces
a natural monomorphism $X^{\sur} \to X$.

\begin{lem}\label{lem-pro-identity}
Let $X=(X_{d})$ be a pro-object.
An  endomorphism $\alpha:X \to X$  is the identity
if and only if for each $d$, there exists $d' \ge d$
such that the bonding morphism $X_{d'} \to X_{d}$ represents $\alpha$.
\end{lem}

\begin{proof}
It follows from the definition of procategory.
\end{proof}

\begin{prop}\label{prop-ML-sur-isom}
For a Mittag-Leffler pro-object $X=(X_{d})$,
the natural morphism $X^{\sur} \to X$
is an isomorphism. 
\end{prop}

\begin{proof}
We construct the inverse of  $X^{\sur} \to X$ as follows: For the pairs $(d,d') \in D\times D$ such that
$\Im(X_{d'} \to X_{d})= X^{\sur}_{d} $, the projections $X_{d'} \to X^{\sur}_{d}$
represent a morphism $X \to X^{\sur}$. From the preceding lemma,
this is the inverse of the $X^{\sur} \to X$.
\end{proof}

\begin{prop}\label{prop-ML-isom}
A pro-object which is isomorphic to a Mittag-Leffler pro-object
is Mittag-Leffler. Thus being Mittag-Leffler is stable under isomorphisms.
\end{prop}

\begin{proof}
Let $X=(X_{d})$ be a pro-object which is isomorphic to a Mittag-Leffler pro-object.
From the preceding proposition, $X$ is also isomorphic to an epi
pro-object, say $Y=(Y_{e})$. 
Let $\phi:X \to Y$ and $\psi:Y \to X$ be isomorphisms which are the inverse
to each other. We now fix a member $X_{d}$ of $(X_{d})$. 
If $\psi_{d}^{e} :Y_{e} \to X_{d}$ 
represents $\psi$, then put $X'_{d}:=\Im (\psi_{d}^{e})$, which is independent
of  the choice of $\psi_{d}^{e}$, because $Y$ is epi.
For every $d' \ge d$ and for any morphism $Y_{e'} \to X_{d'}$ representing
$\psi$, we have
\[
\Im(X_{d'} \to X_{d}) \supseteq \Im (Y_{e'} \to X_{d'} \to X_{d}) = X'_{d}.
\]
On the other hand, for the above $\psi_{d}^{e} :Y_{e} \to X_{d}$,
and for a morphism $X_{d''} \to Y_{e}$ representing $\phi$,
we have
\[
X'_{d} \subseteq \Im (X_{d''} \to X_{d} )= \Im (X_{d''} \to Y_{e} \to X_{d} ) \subseteq \Im ( Y_{e} \to X_{d} ) =X'_{d}.
\]
Hence 
\[
 \Im (X_{d''} \to X_{d} ) =X'_{d} \subseteq \bigcap _{d' \ge d} \Im(X_{d'} \to X_{d})
  \subseteq  \Im (X_{d''} \to X_{d} ) .
\]  
It follows that $X$ is Mittag-Leffler.
\end{proof}

Thus the essential image of the class of epi pro-objects in
$\pro\bA$ is that of Mittag-Leffler pro-objects.

\begin{lem}\label{lem-surjective}
Let $\phi:M=(M_{d}) \to N=(N_{e})$ be a morphism in $\pro\bA$.
Suppose that $(N_{e})$ is epi.
Then $\phi$ is an epimorphism if and only if every morphism $M_{d} \to N_{e}$
representing $\phi$ is an epimorphism.
\end{lem}

\begin{proof}
The ``if'' is obvious and holds without the assumption that $N$ is epi.
Suppose that $\phi$ is an epimorphism and 
some morphism $\psi:M_{d} \to N_{e}$ is not an epimorphism.
Put $N'_{e}:= \Im(\psi)$. Then the natural morphism $N\to N_{e}$ is an epimorphism
and so is the composite $M \to N \to N_{e}$.
But the last morphism factors as $M \to N'_{e} \to N_{e}$,
and is not an epimorphism, a contradiction.
\end{proof}

\begin{lem}\label{lem-epi-zero}
Let $M=(M_{d}) \in \pro\bA$.
Then $M \cong 0$ if and only if  $M$ is Mittag-Leffler and for every $d$,
$M^{\sur}_{d} =0$. 
\end{lem}

\begin{proof}
If $M$ is Mittag-Leffler and for every $d$,
$M^{\sur}_{d} =0$, then $0=M^{\sur} \cong M$.
Conversely if $M \cong 0$, then $M$ is Mittag-Leffler
and $M^{\sur} \cong M \cong 0$.
Hence there exists an isomorphism $ 0 \to M^{\sur}$. 
From Lemma \ref{lem-surjective}, for every $d$, $M^{\sur}_{d}=0$.
\end{proof}

\begin{cor}
Let $A$ be a proring.
Then a  sequence 
\[
0 \to L \to M \to N \to 0
\]
of $A$-promodules is exact if and only if it is exact as a sequence of
abelian progroups.
\end{cor}

\begin{proof}
Consider  a level morphism $(\phi_{d}:X_{d } \to Y_{d}) $ of $A$-modules
and the induced morphism $\phi :(X_{d}) \to (Y_{d}) $ of $A$-promodules.
Then both in the category of $A$-promodules and that of abelian progroups,
we have $ \Ker (\phi) \cong (\Ker (\phi_{d}))$.
Hence 
\begin{align*}
& \text{$\phi$ is a monomorphism in the category of $A$-promodules} \\
&  \Leftrightarrow  \Ker (\phi) \cong 0 \text{ as an $A$-promodule} \\
& \Leftrightarrow  \Ker (\phi) \cong 0 \text{ as an abelian progroup  (the preceding lemma)} \\
& \Leftrightarrow \text{$\phi$ is a monomorphism in the category of abelian progroups}. 
\end{align*}
A similar statement holds for the epimorphism. Now the corollary is obvious.
\end{proof}

\section{Formal schemes as proringed spaces}\label{sec-formal-sch-pro}

In this section, we define formal schemes as proringed spaces.

\subsection{Sheaves on a qsqc basis}

\begin{defn}
Let $X$ be a topological space.
We say that $X$ is {\em quasi-separated} if 
for any two quasi-compact open 
$U,V \subseteq X$, $U \cap V$ is quasi-compact.
We say that $X$ is {\em qsqc}
if  it is quasi-separated and quasi-compact.
A {\em qsqc basis} is a basis of open subsets consisting of qsqc open subsets.
\end{defn}

For instance, qsqc are the underlying topological space of an affine scheme
and a Noetherian topological space.

Throughout the paper, we assume that every topological space has a  qsqc basis. 
For instance, the underlying topological space of any scheme 
satisfies this.  

\begin{defn}
Let $X$ be a topological space and $\fB$ a basis of open subsets of $X$.
For  open $U \subseteq X$, a (necessarily open) covering $U = \bigcup  U_{i}$
is called a {\em $\fB$-covering} if for every $i$, $U_{i} \in \fB$.
A {\em bicovering} of an open subset $U \subseteq X$ consists
of a covering $U = \bigcup  U_{i}$ and coverings $U_{i}\cap U_{j } = \bigcup
U_{ijk}$ for each two distinct indices $i ,j$, which we denote by $\{U_{i},U_{ijk}\}$.
A  bicovering $\{U_{i},U_{ijk}\}$ is said to be {\em finite} if
it consists of finite coverings, that is, $\{U_{i},U_{ijk}\}$ is a finite set.
A  bicovering $\{U_{i},U_{ijk}\}$ is called a {\em $\fB$-bicovering} if
$U_{i},U_{ijk} \in \fB$ for every $i$ and $(i,j,k)$.
\end{defn}

Let $X$ be a topological space and $\fB$ a qsqc basis.
Then every qsqc open $U \subseteq X$ has a finite $\fB$-bicovering.
Let  $\bK$  be a category which admits finite projective limits.
We say that a {\em presheaf} in $\bK$ on $\fB$ is 
a contravariant functor $\fB \to \bK$. Here we think of $\fB$
as a category so that the only morphisms are the inclusion maps. 
A presheaf $\cF$ is called a {\em sheaf} if for every $U \in \fB$
and its every finite $\fB$-bicovering  $\{U_{i},U_{ijk}\}$, 
the sequence
\begin{equation*}\label{sheaf-sequence}
0 \to \cF(U) \to \prod \cF(U_{i}) \rightrightarrows \prod \cF(U_{ijk})
\end{equation*}
is exact. 
When $\fB$ consists of {\em all} qsqc open subsets, we just say that 
$\cF$ is a {\em presheaf} or {\em sheaf} on $X$ respectively.
For any qsqc basis $\fB$ and a sheaf $\cF$ on $\fB$,
we can extend $\cF$ to all qsqc open subsets so that it becomes a sheaf on $X$:
For any qsqc open $U \subseteq X$, take a finite $\fB$-bicovering $\{U_{i},U_{ijk}\}$ of $U$
and put 
\[
\cF(U) := \varprojlim_{V \in \{U_{i},U_{ijk}\}} \cF(V).
\]
It is easy to show that the $\cF(U)$ defined in this way is independent of the 
choice of bicovering and the extended $\cF$ is a sheaf on $X$.

\begin{rem}\label{rem-only-qsqc}
An advantage of considering only qsqc open subsets is that
we need to consider  only {\em finite} bicoverings,
which corresponds to considering {\em finite} projective limits.
For instance, the following proposition is not true if
we consider values at all open subsets.
\end{rem}

\begin{prop}\label{prop-sheaf-pro}
Let $X$ be a topological space and $\bC$ a category having finite 
projective limits.
Then  a sheaf in $\bC$ on $X$ is  a sheaf in $\pro\bC$ as well.
\end{prop} 

\begin{proof}
It is because the inclusion $\bC \hookrightarrow \pro\bC$ is exact
(Proposition \ref{prop-incl-exact}).
\end{proof}

\begin{prop}\label{prop-gluing-sheaves}
Let $X=\bigcup_{i \in I} U_{i}$ be an open covering of a topological space,
and  $\cF_{i}$, $i \in I$, sheaves in $\bK$ on $U_{i}$ respectively with a gluing data,
Then we can glue them to obtain a sheaf $\cF$ on $X$ with $\cF|_{U_{i}} =\cF_{i}$.
\end{prop}

\begin{proof}
We first define $\cF$ as a sheaf on the qsqc
open subsets $V \subseteq X$ such that $V \subseteq U_{i}$ for some $i$.
Then we uniquely extend it to all qsqc open subsets as above.
\end{proof}

\subsection{Locally admissibly proringed spaces}\label{subsec-locally-admissibly}

\begin{defn}
\begin{enumerate}
\item 
An {\em admissible system of (affine) schemes} 
is a directed inductive system $(X_{d})$ of (affine) schemes
such that every bonding morphism $X_{d} \to X_{d'}$ is  a bijective closed immersion.
\item 
A proring $(A_{d})$ is said to be {\em admissible}
if every bonding morphism $A_{d'} \to A_{d}$ is surjective
and induces an isomorphism $(A_{d'})_{\red} \to (A_{d})_{\red}$ of
the associated reduced rings, or equivalently
if  $(\Spec A_{d})$ 
is an admissible system of affine schemes.
For an admissible proring $A=(A_{d}) $, we define the associated reduced
ring $A_{\red}$ by $A_{\red} := A_{d,\red}$ for any $d$.
\end{enumerate}
\end{defn}

A morphism $\alpha:A \to B$ of admissible rings induces a morphism
$\alpha_{\red}:A_{\red} \to B_{\red}$ in an obvious way. 
If $\alpha$ is an isomorphism, so is $\alpha_{\red}$.

\begin{defn}
A {\em proringed space} is a topological space $X$ endowed with 
a sheaf $\cO_{X}$ of prorings, which is called the {\em structure sheaf}. 
A  proringed space $X$ is called an {\em admissibly proringed space}
if there exists a qsqc basis $\fB$ such that for every $U \in \fB$, 
$\cO_{X}(U)$ is an admissible proring (modulo isomorphisms).
\end{defn}

For an admissibly proringed space $X$, if $\fB$ is as above,
we have the sheaf of rings on $\fB$,
\[
\fB \ni U \mapsto (\cO_{X}(U))_{\red}.
\]
Extending it to all qsqc open subsets, we obtain a sheaf of rings on $X$
and denote it by $\cO_{X_{\red}}$. Then we denote by $X_{\red}$
the ringed space obtained by replacing $\cO_{X}$ with $\cO_{X_{\red}}$.

\begin{defn}
A {\em locally admissibly proringed space} is an admissibly proringed
space $X$ such that for every $x \in X$, $\cO_{X_{\red},x}$ is a local ring.
A {\em morphism} $\phi:Y \to X$ of locally admissibly proringed spaces
consists of a continuous map $\phi :Y \to X$ denoted by the same symbol
and a $\phi $-morphism $\phi^{*} :\cO_{X} \to \cO_{Y}$
such that for every $y$, the induced map $\cO_{X_{\red},\phi(y)} \to \cO_{Y_{\red},y}$
is a local homomorphism.
\end{defn}

In the definition, a $\phi$-morphism
$ \cO_{X} \to \cO_{Y}$ means a compatible collection
of morphisms $\cO_{X}(U) \to \cO_{Y}(V)$, where 
 $U$ and $V$ run over the qsqc open subsets $U \subseteq X$
 and $V \subseteq Y$ with $\phi(V) \subseteq U$.

\subsection{Formal spectra}

\begin{defn}
For an admissible proring $A=(A_{d})$, we define the {\em formal spectrum}
 $\Spf A$ as a locally admissibly proringed space as follows:
As a topological space, define 
\[
\Spf A:= \Spec A_{\red} =\Spec A_{d} , \ \forall d.
\]
We have a directed projective system $(\cO_{\Spec A_{d}})$ of sheaves of rings.
Now we define $\cO_{\Spf A}:= \prolim \cO_{\Spec A_{d}}$ (defined on qsqc open subsets).
\end{defn}

As easily checked, it is actually a locally admissibly proringed space.
We have $(\Spf A)_{\red} =\Spec A_{\red}$.
If $A$ is a ring, then it is also a admissible proring and $\Spf A= \Spec A$.

Let $A=(A_{d})$ be an admissible proring.
For $f \in A_{\red}$, write $D(f):= \Spec A_{\red,f} \subseteq \Spec A_{\red}$,
which is called an {\em distinguished open subset}.
Here $A_{\red,f}$ is the localization of $A_{\red}$ by $f$.
If $f' \in A_{d}$ is a lift of $f$, then $\cO_{\Spec A_{d}} (D(f)) = A_{d,f'}$.
Since  $A_{d,f'}$ is independent of the lift $f'$,
we write 
\[
 A_{d,f} := A_{d,f'} \text{ and } A_{f}:=( A_{d,f})
\]
Then $(D(f), \cO_{\Spf A}|_{D(f)}) = \Spf A_{f}$.

\begin{defn}
Let $(A_{d})$ be an admissible proring.
Then for each point $x \in \Spf A$, we define the {\em stalk} $\cO_{\Spf A,x}$ of $\cO_{\Spf A}$ at $x$
as the proring $( \cO_{\Spec A_{d},x} )$.
\end{defn}

If $\fp \subseteq A_{\red}$ is a prime ideal
and $\fp_{d} \subseteq A_{d}$ is its inverse image,
then 
\[
 \cO_{\Spec A_{d}, \fp} = A_{d,\fp_{d}} =: A_{d,\fp},
\]
where $A_{d,\fp_{d}}$ is the localization of $A_{d}$ with respect to $\fp$, and
\[
\cO_{\Spf A, \fp} = (A_{d,\fp}) =: A_{\fp}.
\]
Every stalk is local in the following sense:

\begin{defn}
An admissible proring $A=(A_{d})$ is said to be {\em local}
if some (and every) $A_{d}$ is a local ring, or equivalently if
$A_{\red}$ is a local ring.
A morphism $\alpha:(A_{d}) \to (B_{e})$ of local admissible rings
is said to be {\em local} if some (and every) morphism
$A_{d} \to B_{e}$ representing $\alpha$ is a local homomorphism,
or equivalently if $\alpha_{\red}$ is a local homomorphism.
\end{defn}

\subsection{Morphisms of formal spectra}

If $\alpha : A=(A_{d}) \to B=(B_{e})$ is a morphism of admissible prorings,
then for every morphism $A_{d} \to B_{e}$ representing it, 
we have the induced morphism of schemes, $\Spec B_{e} \to \Spec A_{d}$.
Then the morphisms $\Spec B_{e} \to \Spec A_{d}$ determines
a morphism $\phi_{\alpha}:\Spf B \to \Spf A$ of locally admissibly proringed spaces,
which is called the {\em morphism induced from $\alpha$}.
For  each $y \in \Spf B$, 
we have the natural local morphism $\cO_{\Spf B,y} \to \cO_{\Spf A,\phi_{\alpha}(x)}$.

\begin{prop}\label{prop-spf}
Let $A$ and  $B$ be admissible prorings and $\phi:\Spf B \to \Spf A$
a morphism of locally admissibly proringed spaces.
Then $\phi$ is the morphism induced from the morphism 
\[
\phi^{*}: A =\cO_{\Spf A}(\Spf A) \to B=\cO_{\Spf B}(\Spf B) .
\]
\end{prop}

\begin{proof}
$\phi$ induces a morphism $\phi_{\red}:\Spec B_{\red} \to \Spec A_{\red}$
of locally ringed spaces, which is induced  from $\phi_{\red}^{*}:A_{\red} \to B_{\red}$.
 In particular,
if $f \in A_{\red}$ and $ g:= \phi_{\red}^{*}(f) \in B_{\red}$,
then $\phi^{-1}(D(f))=D(g)$. 
Hence we have the commutative diagram
\[
\xymatrix{
A \ar[r]^{\phi^{*}} \ar[d] &B \ar[d] \\
A_{f} \ar[r]_{\phi^{*}} & B_{g} .
}
\]
It easily follows from the universality of localization that
 the bottom arrow is uniquely determined by the top one.
This shows the assertion.
\end{proof}

\subsection{Formal schemes}

\begin{defn}
An {\em affine formal scheme} is a locally admissibly proringed space
which is
 isomorphic  to the formal spectrum of some admissible proring.
\end{defn}

\begin{cor}
The category of affine formal schemes is equivalent to the dual
category of the category of admissible prorings.
\end{cor}

\begin{proof}
It follows from Proposition \ref{prop-spf}.
\end{proof}

\begin{defn}
A  locally admissibly proringed space
$(X,\cO_{X})$ is called a {\em formal scheme}
if there exists an open covering
$X=\bigcup U_{i}$ such that for every $i$, $(U_{i},\cO_{X}|_{U_{i}})$
is an affine formal scheme.
\end{defn}

If $X$ is a scheme, then the structure sheaf $\cO_{X}$, which is a
sheaf of rings, is also a sheaf of prorings (Proposition \ref{prop-sheaf-pro}), and $X$ is regarded as
a formal scheme as well.\footnote{It holds because we consider only 
values of sheaves at qsqc open subsets. See Remark \ref{rem-only-qsqc}.} 
Thus the category of schemes is a full subcategory of that of formal schemes.
By abuse of terminology, a formal scheme which is isomorphic to a scheme
is also called a {\em scheme}.

\begin{conv}
For an open subset $U$ of a formal scheme $X$, $(U,\cO_{X}|_{U})$
is again a formal scheme, which is called an {\em open formal subscheme}.
Unless otherwise noted, we identify an open formal subscheme with 
its underlying open subset. We say that an open subset $U \subseteq X$ is {\em affine}
if it is an affine formal scheme.
If $U \cong \Spf A$, then we write $\Spf A \subseteq X$.
\end{conv}

\begin{prop}
We can glue formal schemes along open formal subschemes.
\end{prop}

\begin{proof}
We can first glue them as a topological space and then 
glue the structure sheaves as in Proposition \ref{prop-gluing-sheaves}.
\end{proof}

\begin{prop}
Every formal scheme  admits a basis $\fB$ of open subsets
consisting of affine ones, which is a qsqc basis.
\end{prop}

\begin{proof}
Every formal scheme $X$ has an affine covering $X =\bigcup U_{i}$ and each $U_{i}$
has a basis of open subsets consisting of distinguished open subsets.
This proves the proposition.
\end{proof}

\begin{prop}
Let $A$ be an admissible proring. Then $\Spf A$ is a scheme if and only if
$A$ is isomorphic to a ring.
\end{prop}

\begin{proof}
The ``if'' is trivial. Suppose that $X:=\Spf A$ is a scheme.
Then we have a finite open covering $\Spf A = \bigcup U_{i}$
by affine schemes. Put $A_{i}:= \cO_{X}(U_{i})$, which are rings.
For each $i,j$, $A_{ij}:=\cO_{X}(U_{i}\cap U_{j})$ is also a ring.
Then $A$ is isomorphic to the difference kernel of 
natural maps $\prod A_{i} \rightrightarrows \prod A_{ij}$,
which is a ring. It proves the ``only if.''
\end{proof}

\begin{defn}
For a formal scheme $X$ and $x \in X$,
if $\Spf A \subseteq X$ is an affine neighborhood of $x$,
we define the {\em stalk} $\cO_{X,x}$ of $\cO_{X}$ at $x$
to be $\cO_{\Spf A,x}$. (It is clearly independent of the affine neighborhood.)
\end{defn}

\subsection{Fiber products}

The category of formal schemes has fiber products.
For morphisms $Y \to X$ and $Z \to X$ of formal schemes,
we denote the fiber product by $Y \times_{X} Z$.
To show the existence of fiber products,
we only have to consider the case where $X$, $Y$ and $Z$ are affine.
If we write $X= \Spf A$, $Y=\Spf B$ and $Z=\Spf C$,
then since $B \otimes _{A}C$ is the fiber sum,
dually the fiber product $Y \times _{X} Z$ exists and is isomorphic to
$ \Spf B \otimes_{A} C$.

\section{Semicoherent promodules and formal subschemes}\label{sec-semicoh}

In this section, we introduce the notions of
semicoherent promodule and formal subschemes,
and study their basis properties.

\subsection{Semicoherent promodules}\label{subsec-semicoh}

\begin{defn}
Let $X$ be a formal scheme.
An $\cO_{X}$-{\em (pro)module} is a sheaf of abelian (pro)groups, $\cM$,
such that for each qsqc open $U \subseteq X$, $\cM (U)$
is an $\cO_{X}$-(pro)module in a compatible way.

A {\em morphism} $\alpha:\cM \to \cN$ of $\cO_{X}$-(pro)modules
is a compatible system of morphisms, 
\[
\alpha(U):\cM(U) \to \cN(U), \ \text{qsqc } U \subseteq X,
\]
of $\cO_{X}(U)$-(pro)modules.
 
\end{defn}

Let $A=(A_{d})$ be an  admissible proring, $X:=\Spf A$
and $M_{0}$ an $A$-module.
Reindexing $A$, we may suppose that every $A_{d}$ acts on $M$.
Put $X_{d}:= \Spec A_{d}$.
For each $d$, we have the 
$\cO_{X_{d}}$-module $\tilde M_{0}$ associated to $M_{0}$.
As a sheaf of abelian groups, $\tilde M_{0}$ is independent of $d$.
All the $\cO_{X_{d}}$-module structures make $\tilde M_{0}$
an $\cO_{X}$-module.
Next, for an $A$-promodule $M=(M_{e})$, we define
the {\em associated $\cO_{X}$-promodule} $M^{\triangle}$
by $M^{\triangle}:= \prolim  \tilde M_{e}$.
For $f \in A_{\red}$ and a lift $f' \in A_{d}$, 
\[
\tilde M_{0}(D(f)) = M_{0,f'} = M_{0} \otimes _{A_{d}} A_{d,f} =:M_{0,f},
\]
and hence
\[
M^{\triangle} (D(f)) =( M_{e,f}) = M \otimes_{A}A_{f} =:M_{f}.
\]
For $x \in X$, if $\fp \subseteq A_{\red}$ is the corresponding prime ideal,
 the stalk $\tilde M_{0,x}$ is an $\cO_{X, x}$-module and
\[
\tilde M_{0,x} = M_{0} \otimes_{ A} A_{\fp} =: M_{0,\fp}.
\]
Then we define the {\em stalk} $M^{\triangle}_{x}$ by
\[
M^{\triangle}_{x}=M_{\fp}:= M \otimes _{A}A_{\fp} =(M_{e,\fp}),
\]
which is an $\cO_{X,x}$-promodule.

Let $N=(N_{c})$ be another $A$-promodule and $\alpha :M \to N$ a morphism
of $A$-promodules.
The morphisms $\alpha^{e}_{c}:M_{e} \to N_{c}$ representing $\alpha$
determines 
\[
\tilde \alpha^{e}_{c} (U) : \tilde M_{e}(U) \to \tilde N_{c}(U),\ \text{qsqc } U \subseteq X.
\]
These morphisms then determine $M ^{\triangle} (U) \to N ^{\triangle}(U)$, $U \subseteq X$,
and  a morphism $\alpha^{\triangle}:M^{\triangle} \to N ^{ \triangle}$ of $\cO_{X}$-promodules.
Conversely if $\beta :M ^{\triangle} \to N^{\triangle}$ is a morphism of $\cO_{X}$-promodules,
then $\beta = (\beta (X))^{\triangle}$. Indeed, for each $f \in A_{\red}$, 
we have the commutative diagram
\[
\xymatrix{
 M \ar[d] \ar[r] ^{\beta(X)}& N \ar[d] \\
 M_{f} \ar[r]_{\beta(D(f))} & N _{f}.
}
\]
From the universality of localization, $\beta(D(f))$ should be the one
induced from $\beta (X)$, which shows $\beta = (\beta (X))^{\triangle}$.

\begin{defn}
Let $X$ be a formal scheme. 
An $\cO_{X}$-promodule $\cM$ is said to be {\em semicoherent}
if every point of $X$ has an affine neighborhood $\Spf A \subseteq X$
such that $\cM|_{\Spf A} \cong M^{\triangle}$ for some $A$-promodule
$M$.
For a semicoherent $\cO_{X}$-promodule $\cM$ and for $x \in X$,
we define the {\em stalk} $\cM_{x}$ by $M^{\triangle}_{x}$ for $\Spf A$ and
$M$ as above. 
\end{defn}

For a scheme $X$, every quasi-coherent $\cO_{X}$-module is
a semicoherent $\cO_{X}$-module and vice versa:
\[
\text{(Quasi-coherent $\cO_{X}$-modules)} = \text{(Semicoherent $\cO_{X}$-modules)}.
\]
However in general there are much more semicoherent $\cO_{X}$-{\em pro}modules
than  quasi-coherent $\cO_{X}$-modules (see Example \ref{expl-completion}).

\begin{lem} \label{lem-affine-restriction}
Let $A$ be an admissible proring, $M$ an $A$-promodule, 
and $\Spf B \subseteq \Spf A$ an affine open subset.
Then $M^{\triangle} |_{\Spf B} =(M \otimes_{A}B)^{\triangle}$.
\end{lem}

\begin{proof}
Obvious.
\end{proof}

\begin{prop}\label{prop-affine-semicoh}
Let $\cM$ be a semicoherent $\cO_{X}$-promodule
on an affine formal scheme  $\Spf A$.
Then $\cM = (\cM(\Spf A))^{\triangle}$.
\end{prop}

\begin{proof}
There exists a Zariski covering $\phi:\Spf B \to \Spf A$
such that $\phi^{*}\cM = (M')^{\triangle}$ for the $B$-promodule 
$M' := (\phi^{*}\cM)(\Spf B)$.
Put $C:=B \otimes_{A} B$
and let $\psi :\Spf C \to \Spf A$ be the natural morphism, which is again
a Zariski covering. 
Then $\psi^{*}\cM \cong (M'')^{\triangle}$, where
$M'' := (\psi^{*}\cM)(\Spf C)$.

Then we have the exact sequence
of $A$-promodules
\[
0 \to M \to M'_{A} \rightrightarrows M''_{A} .
\]
For each affine open $\Spf D \subseteq \Spf A$, 
since $- \otimes_{A} D$ is exact,
the sequence
\[
0 \to M\otimes_{A} D \to (M'_{A})\otimes_{A} D \rightrightarrows  (M''_{A} )\otimes_{A} D
\]
is exact. It follows that $\cM(\Spf D)=M\otimes_{A} D$,
 which implies the assertion. 
\end{proof}

The following is a direct consequence:

\begin{cor}
For an affine formal scheme $X=\Spf A$, we have the equivalence
\begin{align*}
\textup{($A$-promodules)} & \cong \textup{(Semicoherent $\cO_{X}$-promodules)} \\
  M & \mapsto M ^{\triangle} \\
   \cM(X) & \mapsfrom   \cM .
\end{align*}
\end{cor}

\begin{cor}
For a formal scheme $X$, the category of semicoherent $\cO_{X}$-promodules
is abelian.
\end{cor}

\begin{proof}
From the preceding corollary, for a morphism $\cM \to \cN$
of semicoherent $\cO_{X}$-promodules, its kernel and image
is defined on each affine open subset. Gluing the local ones,
we obtain the globally defined kernel and image.
The rest is easy to check.
\end{proof}

\subsection{Pullback and pushforward}

\begin{defn}
Let $\phi:Y  \to X$ be a morphism of formal schemes
and $\cM$ a semicoherent $\cO_{X}$-promodule.
We define the {\em pullback} $\phi^{*} \cM$, which is a semicoherent
$\cO_{Y}$-promodule, as follows:
If $V \subseteq Y$ is an affine open subset such that 
$\phi(V)$ is contained in an affine open $U \subseteq X$,
then put $(\phi^{*}\cM)(V):= \cM(U) \otimes_{\cO_{X}(U)} \cO_{Y}(V)$,
which is independent of $U$ from Lemma \ref{lem-affine-restriction}.
Such $V$'s form a qsqc basis of $Y$ and $\phi^{*}\cM$ is
a sheaf on this basis. We can now uniquely extend it to all qsqc open 
subsets.
\end{defn}

With the above notation, If $Y=\Spf B$, $X=\Spf A$
and $\cM=M^{\triangle}$, then $\phi^{*}(M^{\triangle} )=(M \otimes _{A}B)^{\triangle}$.

\begin{defn}
A continuous map $\phi :Y \to X$ is said to be {\em qsqc} if for every qsqc open $U \subseteq X$,
$\phi^{-1}(U)$ is qsqc.
\end{defn}

\begin{defn}\label{def-pushforward}
Let $\phi:Y \to X$ be a qsqc morphism of formal schemes and  $\cN$ a semicoherent
$\cO_{Y}$-promodule.  We define the {\em pushforward} $\phi_{*}\cN$,
which is an $\cO_{X}$-promodule,
as follows: For qsqc open $U \subseteq X$, $(\phi_{*}\cN)(U):=\cN(\phi^{-1}(U))_{\cO_{X(U)}}$.
\end{defn}

\begin{lem}
Let $\phi :\Spf B \to \Spf A$ be a morphism of affine formal schemes,
which is always qsqc. For a $B$-promodule $N$,
$\phi_{*}(N^{\triangle})=(N_{A})^{\triangle}$.
\end{lem}

\begin{proof}
Obvious from the definition.
\end{proof}

\begin{prop}
Let the notation be as in Definition \ref{def-pushforward}.
Then $\phi_{*}\cN$ is semicoherent.
In addition, if $\cN$ is an $\cO_{Y}$-module, then $\phi_{*}\cN$ is
an $\cO_{X}$-module.
\end{prop}

\begin{proof}
We may suppose that $X$ is affine, say $X =\Spf A$.
Then $Y$ is qsqc, so there exists an affine bicovering 
$\{U_{i}, U_{ijk}\}$ of $Y$.
Put $N_{i}:= \cN(U_{i})$ and $N_{ijk}:=\cN(U_{ijk})$.
If we denote the restrictions of $\phi$ to $U_{i}$ 
and $U_{ijk}$ also by $\phi$,
then $\phi_{*}(\cN|_{U_{i}}) =((N_{i} )_{A})^{\triangle}$ 
and $\phi_{*}(\cN|_{U_{ijk}}) =((N_{ijk} )_{A})^{\triangle}$,
which are semicoherent.
Since $ \phi_{*} \cN $
is the kernel of a morphism of semicoherent $\cO_{X}$-promodules
\[
 \bigoplus \phi_{*}(\cN|_{U_{i}}) \to \bigoplus \phi_{*}(\cN|_{U_{ijk}}),
\]
$ \phi_{*} \cN $ is semicoherent too, which proves the first assertion.

If $\cN$ is an $\cO_{Y}$-module, then $ \bigoplus \phi_{*}(\cN|_{U_{i}})$
and $\bigoplus \phi_{*}(\cN|_{U_{ijk}})$ are $\cO_{X}$-modules,
and hence so is $\phi_{*} \cN$.
\end{proof}

\begin{prop}
Let $\phi:Y \to X$ be a qsqc morphism of formal schemes,
$\cN$ a semicoherent $\cO_{Y}$-promodule
and $\cM$ a semicoherent $\cO_{X}$-promodule.
Then we have a natural equation
\[
\Hom(\cM, \phi_{*} \cN) = \Hom ( \phi^{*}\cM ,\cN) .
\]
Namely $\phi_{*}$ is the right adjoint of $\phi^{*}$.
Hence $\phi_{*}$ is left exact and $\phi^{*}$ is right exact.
\end{prop}

\begin{proof}
In general, let $A$ be a proring, $B$ an $A$-proalgebra,
$M$ an $A$-promodule,  and $N$ a $B$-promodule.
Then 
\[
 \Hom _{A} (M,N_{A}) = \Hom _{B}(M \otimes_{A}B,N).
\]
Giving a morphism$ \cM \to \phi_{*}\cN$ is equivalent to
giving a compatible system of  $\cO_{X}(U)$-morphisms 
$\cM(U) \to \cN(V)_{\cO_{X}(U)}$ for affine open subsets $U \subseteq X$
and $V \subseteq Y$ with $\phi(V) \subseteq U$.
In turn, it is equivalent to giving a compatible system 
of $\cO_{Y}(V)$-morphisms $\cM(U) \otimes _{\cO_{X}(U)}\cO_{Y}(V) \to \cN(V)$
for such pairs $(U,V)$.
Finally it is equivalent to giving a morphism $\phi^{*}\cM \to \cN$,
which completes the proof.
\end{proof}

\subsection{Mittag-Leffler $\cO_{X}$-promodules}\label{subsec-ML-sheaf}

\begin{defn}
Let $X$ be a formal scheme and $\cM$ a semicoherent $\cO_{X}$-promodule.
We say that $\cM$ is {\em Mittag-Leffler} if
every $x \in X$ has an affine neighborhood $U \subseteq X$
such that $\cM(U)$ is Mittag-Leffler.
\end{defn}

\begin{prop}
Suppose that $X=\Spf A$ is an affine formal scheme and
that $\cM$ is a Mittag-Leffler semicoherent $\cO_{X}$-promodule.
Then $\cM(X)$ is Mittag-Leffler
\end{prop}

\begin{proof}
Write $M=(M_{e}):= \cM(X) $.
There exists a finite open covering $X= \bigcup D(f_{i})$ such
that the $M_{f_{i}}$ are Mittag-Leffler.
Then for each $i$,
the projective system $(\tilde M_{e} |_{D(f_{i})})$ of sheaves
satisfies the Mittag-Leffler condition and so does $(\tilde M_{e})$.
Therefore $M$ is Mittag-Leffler.
\end{proof}

\begin{prop}\label{prop-quot-ML}
Let $\phi:\cF \to \cG$ be an epimorphism of
semicoherent $\cO_{X}$-promodules. If $\cF$ is Mittag-Leffler, then so is $\cG$.
\end{prop}

\begin{proof}
Obvious.
\end{proof}

\begin{prop}
Let $\phi:Y \to X$ be a morphism of formal schemes
and $\cM$ a Mittag-Leffler semicoherent $\cO_{X}$-promodule.
Then  $\phi^{*} \cM$ is also Mittag-Leffler.
\end{prop}

\begin{proof}
It follows from the construction of pullback.
\end{proof}

\subsection{Stalks and exactness}

Let $X$ be a formal scheme.
Given a morphism $\alpha:\cF \to \cG$ of semicoherent
$\cO_{X}$-promodules,  for each $x \in X$,
we have the induced morphism
$\alpha _{x}:\cF_{x} \to \cG_{x}$ of stalks.
Then
\[
 \Ker ( \alpha _{x}) = \Ker (\alpha)_{x} , \ \Im ( \alpha _{x}) = \Im (\alpha)_{x}.
\]
As a consequence, we obtain the following:

\begin{cor}
Let
\[
0 \to \cF \to \cG \to \cH \to 0
\]
be an exact sequence of semicoherent $\cO_{X}$-promodules.
Then  for every $x \in X$, 
the induced sequence of stalks
\[
0 \to \cF_{x} \to \cG_{x} \to \cH_{x} \to 0,
\]
is exact.
\end{cor}

The converse of the corollary holds only under some condition:

\begin{prop}\label{prop-exact-stalk}
Let  $\bD$ be an abelian subcategory of the category
of semicoherent  $\cO_{X}$-promodules.
Suppose that if $\cM \in \bD$ has Mittag-Leffler stalks, then $\cM$ is Mittag-Leffler.
Then a sequence
\[
0 \to \cF \to \cG \to \cH \to 0
\]
 in $\bD$ is exact if for every $x \in X$, 
the induced sequence of stalks
\[
0 \to \cF_{x} \to \cG_{x} \to \cH_{x} \to 0,
\]
is exact.
\end{prop}

\begin{proof}
Put $\cK := \Ker (\cF \to \cG)$. By assumption the stalks of $\cK$
are zero, in particular, Mittag-Leffler. So $\cK$ is Mittag-Leffler.
From the following lemma, $\cK=0$. Hence the $\cF \to \cG$ is a monomorphism.
The rest can be proved similarly. 
\end{proof}

\begin{lem}
Let $X$ be a formal scheme and $\cF $ a semicoherent $\cO_{X}$-promodule.
Then $\cF=0$ if and only if $\cF$ is Mittag-Leffler and for every $x \in X$,
$\cF_{x}=0$.
\end{lem}

\begin{proof}
The ``only if'' is trivial. We prove the ``if.'' We may suppose that
$X=\Spf A$. Then $F=(F_{d})_{d \in D}:= \cF(X)$ is a Mittag-Leffler $A$-promodule
and so we may suppose also that $F$ is epi.
Then for each prime ideal $\fp \subseteq A_{\red}$, 
$\cF_{x} =(F_{d,\fp})_{d \in D} =0$. From Lemma \ref{lem-epi-zero}, for each $d \in D$,
$F_{d,\fp}=0$. Hence $F_{d}=0$ and $F=0$.
\end{proof}

\begin{prop}\label{prop-ML-epi-stalk}
Let $\phi:\cF \to \cG$ be a morphism of semicoherent $\cO_{X}$-promodules.
Suppose that $\cG$ is Mittag-Leffler and that 
for every $x \in X$, $\phi_{x}:\cF_{x} \to \cG_{x}$
is  an epimorphism.
Then $\phi$ is an epimorphism. 
\end{prop}

\begin{proof}
Put $\cH := \Coker ( \phi)$, which is Mittag-Leffler from Proposition \ref{prop-quot-ML}.
Moreover from the assumption, its stalks are zero. 
So, from the preceding lemma,  $\cH =0$ and
$\phi$ is an epimorphism.
\end{proof}

\subsection{Formal subschemes}

\begin{defn}
Let $X$ be a formal scheme. 
A {\em semicoherent proideal sheaf} on $X$ is a
semicoherent $\cO_{X}$-subpromodule of $\cO_{X}$.
\end{defn}

For a semicoherent proideal sheaf $\cI$ on $X$, the quotient $\cO_{X}$-promodule
$\cO_{X}/\cI$ is naturally regarded as a sheaf of prorings.

\begin{defn}
Let $\cI \subseteq \cO_{X}$ be a semicoherent proideal sheaf
and $Y := \{ x \in X| (\cO_{X}/\cI)_{x} \ne 0 \}$ its support. 
We say that the subspace $Y \subseteq X$ endowed with the sheaf
$ \cO_{X}/\cI$ is a {\em closed formal subscheme} of $X$
if it is a formal scheme.
(Unlike the scheme case,  $(Y, \cO_{X}/\cI)$ is not a priori a formal scheme.)
Then we say that $\cI$ is the {\em defining ideal sheaf} of $Y$.
If $Y$ is even a scheme, then we call it a {\em closed subscheme}.
\end{defn}

For a formal scheme $X$, the $X_{\red}$ defined in \S \ref{subsec-locally-admissibly} is a closed subscheme of $X$.

\begin{lem}
If $Y$ is a closed formal subscheme of a formal scheme $X$, 
then $Y_{\red}$ is a closed subscheme
of $X_{\red}$. In particular, $Y$ is set-theoretically a closed subset of $X$.
\end{lem}

\begin{proof}
Since the problem is local, we may suppose that $X=\Spf A$ for some
admissible proring.
Then put $I := \cI(X)$, a proideal of $A$.
The quotient ring $B:=A/I$ is an admissible proring up to isomorphisms.
From Lemma \ref{lem-surjective}, the epimorphism $A \to B$ (of $A$-promodules)
induces a surjection $A_{\red} \to B_{\red}$, and 
a closed immersion $Y _{\red} \to X_{\red}$, which proves the assertion.
\end{proof}

\begin{expl}\label{expl-completion}
For a scheme $ X $ and a closed subscheme $Y$, 
we define  the {\em completion} $ X_{/Y}$ of $X$ along $Y$,
which is a closed formal subscheme of $X$, as follows:
Let $\cI \subseteq \cO_{X}$ be the defining ideal sheaf of $Y$.
For each $n \in \NN$, we have an epimorphism $\cO_{X} \to \cO_{X}/\cI^{n}$.
Put $\cO:= \prolim \cO_{X}/\cI^{n}$.
Then the natural morphism $\cO_{X} \to \cO$ is an epimorphism
and we have $\cO = \cO_{X}/\cJ$ for some semicoherent proideal sheaf
$\cJ \subseteq \cO_{X}$. Let $Y_{n} \subseteq X$ be the closed subscheme
with the defining ideal $\cI^{n}$ and $\underline{Y}$ the underlying topological
space of $Y$. Then the support of $\cO$ is $\underline{Y}$ and 
the proringed space $(\underline{Y},\cO)$ is a formal scheme, which is
isomorphic to $\varinjlim _{n} Y_{n}$. We write $X_{/Y} := (\underline{Y},\cO)$.
The $\cO_{X}$-promodule $\cO$ is not generally an $\cO_{X}$-module.
\end{expl}

\begin{defn}\label{defn-sub}
A closed formal subscheme of an open formal subscheme of $X$ is
called a {\em formal subscheme} of $X$.
We identify formal subschemes $Y$ and $Z$ of $X$ if
the natural morphisms $Y \to X$ and $Z \to X$ are isomorphic.
A {\em subscheme} of $X$ is a formal subscheme of $X$ which is a scheme.
 
A morphism $W \to X$ of formal schemes is called a {\em (closed, open) immersion}
if it is an isomorphism onto a (closed, open) formal subscheme of $X$.
\end{defn}

\begin{prop}
Let $\phi:Y \to X$ be a morphism of formal schemes.
Then $\phi$ is an immersion if and only if $\phi$ is, as a continuous map,
a homeomorphism onto a locally closed subset of $X$, and 
 for every $y \in Y$,
the natural morphism $\cO_{X,\phi(y)} \to \cO_{Y,y}$ is an epimorphism (of
$\cO_{X,\phi(y)}$-promodules).
\end{prop}

\begin{proof}
The ``only if'' part is obvious and we prove the ``if'' part.
Take an open formal subscheme $U \subseteq X$ such that
 $\phi(Y)$ is a closed subset of $U$.
Then the map $\phi:Y \to U$ is qsqc and $\phi_{*}\cO_{Y}$ is well-defined as
a sheaf on $U$, which is clearly Mittag-Leffler.
From the assumption and Proposition \ref{prop-ML-epi-stalk}, 
the natural morphism $\cO_{U} \to \phi_{*} \cO_{Y}$
is an epimorphism of semicoherent $\cO_{U}$-promodules.
Its kernel is a proideal sheaf and defines a closed subscheme $Y' \subseteq U$.
Then $\phi$ is equal to the composite $Y \cong Y' \hookrightarrow U \hookrightarrow X$
and hence an immersion.
\end{proof}

\begin{prop}\label{prop-immersion-basechange}
Let $Y\to X$ and $W  \to X$ be morphisms of formal schemes.
If the $Y \to X$ is an immersion (resp.\ closed immersion, open immersion), then so is the natural morphism
$W \times_{X} Y \to W$.
\end{prop} 

\begin{proof}
The assertion follows form the preceding proposition and 
the fact that the tensor product is right exact.
\end{proof}

\subsection{Extension of a promodule on an open subset}

The following two propositions generalizes
\cite[Prop.\ 9.4.2 and 9.5.10]{EGA} to formal schemes.

\begin{prop}\label{prop-ext}
Let $X$ be a formal scheme and $U \subseteq X$ an open subset
such that the inclusion map $\iota : U \hookrightarrow X$ is  quasi-compact, hence qsqc.
Let $\cN$ be a semicoherent
 $\cO_X$-promodule, $\cM \subseteq \cN|_U $ a semicoherent $\cO_U$-subpromodule,
and  $\cQ:= \cN|_U /\cM$ the quotient $\cO_{U}$-promodule.
Then there exists the largest  semicoherent 
$\cO_X$-subpromodule $\bar \cM \subseteq \cN$
such that $ \bar \cM |_U = \cM $. Correspondingly there exists the smallest
semicoherent quotient $\cO_X$-promodule $ \cN\twoheadrightarrow  \bar \cQ$
such that $ \bar \cQ |_U = \cQ $.
\end{prop}

\begin{proof}
We have a natural morphism $\phi : \cN \to \iota_* \cQ$.
Then $\bar \cM := \Ker (\phi)$ and $\bar \cQ := \Im (\phi)$
have the desired properties. 
\end{proof}

\begin{prop}\label{prop-closure}
Let $X$ be a formal scheme and $Y \subseteq X$
a formal subscheme. 
Suppose that the inclusion map $ Y \hookrightarrow X$ is quasi-compact, hence
qsqc.
Then there exists a smallest 
closed formal subscheme $\bar Y \subseteq X$ which
contains $Y$ as an open formal subscheme.
If in addition $Y$ is a scheme, then so is $\bar Y$.
\end{prop}

\begin{proof}
We can construct $\bar Y$ as follows: The underlying topological space of $\bar Y$
is the set-theoretic closure of $Y$. If $\iota : Y \hookrightarrow X$ denotes
the inclusion map, then the structure sheaf of $\bar Y$ is 
the image of $\cO_{X} \to \iota_{*} \cO_{Y} $. 

We now check that the above construction gives a closed formal subscheme of $X$.
To do this, we may suppose that $X$ is affine, and hence  $Y$ is qsqc.
Let $(X_{d})$ be an admissible system of affine scheme with $X=\varinjlim X_{d}$
and put  $Y_{d}:= Y \cap X_{d}= Y \times _{X} X_{d}$.
Then $(Y_{d})$ is an admissible system of schemes with $Y = \varinjlim Y_{d}$.
Now we can define the closure $\bar Y_{d}$ of $Y_{d}$ in $X_{d}$
so that the structure sheaf of $\bar Y_{d}$ is
the image of $\cO_{X_{d}} \to (\iota_{d})_{*} \cO_{Y_{d}}$,
where $\iota_{d}$ denotes the inclusion morphism $Y_{d} \hookrightarrow X_{d}$.
It is now easy to see that $\bar Y= \varinjlim \bar Y_{d}$
and that $\bar Y$ is a closed formal subscheme of $X$.

The second assertion of the proposition follows from the construction.
\end{proof}

\begin{defn}
With the notation in the preceding proposition, we call $\bar Y$ 
the {\em closure} of $Y$ in $X$.
\end{defn}

It also follows from the construction that if $Y$ admits an  admissible
system $(Y_{d})$ of schemes with $Y = \varinjlim Y_{d}$,
and if $\bar Y_{d}$ denotes the closure of $Y_{d}$
 in $X$, then $\bar Y = \varinjlim \bar Y_{d}$.

\subsection{Semicoherent promodules on a qsqc formal scheme}

\begin{prop}
For a qsqc formal scheme $X$, we have a natural equivalence
\begin{align*}
&\textup{(Semicoherent $\cO_{X}$-promodule)} \\
& \cong \pro\textup{(Semicoherent $\cO_{X}$-module)} .
\end{align*}
\end{prop}

\begin{proof}
If $(\cM_{d})$ and $(\cN_{e})$ are directed projective systems of
semicoherent $\cO_{X}$-modules, then it is easy to see
\[
 \Hom (\prolim \cM_{d}, \prolim \cN_{e}) = \varprojlim_{e} \varinjlim_{d} \Hom ( \cM_{d},\cN_{e} ).
\]
So it suffices to show that for every semicoherent $\cO_{X}$-promodule 
$\cM$, there exists a directed projective system  $(\cM_{d})$ of
semicoherent $\cO_{X}$-modules with $\cM=\prolim \cM_{d}$.
Let $\{U_{i},U_{ijk}\}$ be a finite affine bicovering of $X$.
If $\iota_{i} :U_{i} \hookrightarrow X$ denotes the inclusion, 
we put $\cM_{U_{i}} : = (\iota_{i})_{*}(\cM|_{U_{i}})$
and similarly define $\cM_{U_{ijk}}$.
We set $\cM' := \prod \cM_{U_{i}}$ and $\cM'' := \prod \cM_{U_{ijk}}$.
Then there exist directed projective systems $(\cM'_{d})$
and $(\cM''_{e})$ of semicoherent $\cO_{X}$-modules
such that $\cM' = \prolim \cM'_{d}$ and $\cM'' = \prolim \cM''_{e}$ respectively.
Then reindexing $(\cM'_{d}) $ and $(\cM''_{e})$,
we may suppose that the natural morphisms $\cM' \rightrightarrows \cM''$
are represented by level morphisms $(\cM'_{d} \rightrightarrows \cM'' _{d})$.
For each $d$, if we put $\cM_{d}$ to be the difference kernel of $\cM'_{d} \rightrightarrows \cM'' _{d}$,
then $\cM=\prolim \cM_{d}$, which completes the proof.
 \end{proof}

\begin{prop}
Every qsqc formal scheme is the inductive limit of some admissible system of schemes.
\end{prop}

\begin{proof}
Let $X$ be an arbitrary qsqc formal scheme.
Take a finite affine covering $X=\bigcup_{i \in I} U_{i}$.
For each $i$, $U_{i}$ is by definition the inductive limit of some admissible system 
of affine schemes, say $(U_{i,d})$.
Without loss of generality, we may suppose that the index sets of
$(U_{i,d})$, $i \in I$, are equal.
We denote $\bar U_{i,d}$ the closure of $U_{i,d}$, which is a closed
subscheme of $X$.
For each $d$, define $X_{d}$ to be the scheme-theoretic union
of $\bar U_{i,d}$, $i \in I$, whose defining proideal is 
the intersection of those of $\bar U_{i,d}$.
Then $X_{d}$ is a subscheme of $X$ and $(X_{d})$ is an admissible system
such that $X=\varinjlim X_{d}$.
\end{proof}

\subsection{Locally ind-Noetherian formal schemes}\label{subsec-ind-Noeth}

\begin{defn}
An admissible proring $(A_{d})$ is said to be {\em pro-Noetherian}
if every $A_{d}$ is Noetherian. 
\end{defn}

From Lemma \ref{lem-surjective}, the pro-Noetherian property for an admissible
proring depends only on its isomorphism class in the category of admissible prorings.

\begin{defn}
A formal scheme $X$ is said to be {\em locally ind-Noetherian}
if every $x \in X$ admits an affine neighborhood $x \in \Spf A \subseteq X$
with $A$ pro-Noetherian.
\end{defn}

\begin{prop}
Let $A$ be an admissible proring such that 
$\Spf A$ is locally ind-Noetherian.
Then $A$ is pro-Noetherian.
\end{prop}

\begin{proof}
If we write $A=(A_{d})$, then $\Spec A_{d}$ is a locally Noetherian
scheme, and hence $A_{d}$ is Noetherian, and  $A$ is pro-Noetherian.
\end{proof}

\begin{defn}
Let $A=(A_{d})$ be a pro-Noetherian admissible ring.
An $A$-module $N$ is said to be {\em Noetherian}
if the set of $A$-submodules of $N$ satisfies the maximal condition,
or equivalently if for some (and every) $d$ such that $A_{d}$ acts on $N$,
$N$ is a Noetherian $A_{d}$-module.
An $A$-promodule $M=(M_{e})$ is said to be {\em pro-Noetherian}
if for every $e$, $M_{e}$ is a Noetherian $A$-module.
\end{defn}

Note that the pro-Noetherian property for promodules is not invariant under isomorphisms.
However if $M$ and $N$ are {\em epi} $A$-promodules isomorphic to each other
and if $M$ is pro-Noetherian, then so is $N$.

It is clear that the category of pro-Noetherian $A$-promodules
is abelian. In particular, every $A$-subpromodule
and quotient $A$-promodule of a pro-Noetherian $A$-promodule
are again pro-Noetherian up to isomorphisms.

\begin{defn}
Let $X$ be a locally ind-Noetherian formal scheme.
A semicoherent $\cO_{X}$-(pro)module $\cM$ is said to be {\em locally 
(pro-)Noetherian} if every point $x \in X$ admits an affine neighborhood
$U=\Spf A $ such that $\cM|_{U} \cong M ^{\triangle}$ for 
some (pro-)Noetherian $A$-(pro)module $M$.
\end{defn} 

Again it is clear that the category of locally Noetherian $\cO_{X}$-modules
and that of locally pro-Noetherian $\cO_{X}$-promodules
are abelian.

\begin{prop}
 Every formal subscheme of a locally ind-Noetherian formal scheme is again
locally ind-Noetherian.
\end{prop}

\begin{proof}
Obvious. 
\end{proof}

\begin{lem}
Let $A$ be a pro-Noetherian admissible proring, $X:= \Spf A$ and 
$\cM$ a locally Noetherian semicoherent  $\cO_{X}$-module.
Then $\cM(X)$ is a Noetherian $A$-module.
\end{lem}

\begin{proof}
Write $A=(A_{d})$. Then $\cM$ is a coherent sheaf on some $\Spec A_{d}$.
So $\cM(X)$ is Noetherian.
\end{proof}

\begin{prop}
Let $A$ be a pro-Noetherian admissible proring, $X:= \Spf A$ and 
$\cM$ a locally
pro-Noetherian semicoherent and Mittag-Leffler $\cO_{X}$-promodule.
Then $\cM(X)$ is a pro-Noetherian $A$-promodule up to isomorphisms.
\end{prop}

\begin{proof}
Let $M = (M_{e})$ be an epi $A$-promodule with $\cM(X) \cong M$
and $X=\bigcup U_{i}$, $U_{i}=\Spf A_{i}$,
 an affine covering such that for each $i$,
$\cM(U_{i})$ is a pro-Noetherian $A_{i}$-promodule up to isomorphisms.
Then for each $i$, being epi, 
$( \tilde M_{e}(U_{i}))$ is actually a pro-Noetherian $A_{i}$-promodule.
Hence every $\tilde M_{e}$ is locally Noetherian, and $M_{e}$ is Noetherian.
We conclude that $(M_{e})$ is pro-Noetherian, which completes the proof.
\end{proof}

\begin{cor}
Let $A$ be a pro-Noetherian admissible ring, $X:=\Spf A$
and  $\bD$ the smallest
abelian full subcategory of the category of locally pro-Noetherian
$\cO_{X}$-promodules which contains all Mittag-Leffler and locally
pro-Noetherian $\cO_{X}$-promodules.
Then for every $\cM \in \bD$, $\cM(X)$ is a pro-Noetherian $A$-promodule
up to isomorphisms. In particular, for every
 semicoherent proideal sheaf $\cI \subseteq \cO_{X}$,
 $\cI(X)$ is pro-Noetherian up to isomorphisms.
\end{cor}

\begin{proof}
The property that $\cM(X)$ is pro-Noetherian up to isomorphisms
is stable under taking direct sums, subobjects and quotient objects.
This proves the corollary.
\end{proof}

\begin{prop}
Let $X$ be a locally ind-Noetherian formal scheme and 
$\cM$ a locally pro-Noetherian $\cO_{X}$-promodule.
Suppose that for every $x \in X$, the stalk $\cM_{x}$ is Mittag-Leffler.
Then $\cM$ is Mittag-Leffler.
\end{prop}

\begin{proof}
Without loss of generality, we may suppose that $X=\Spf A$
for an admissible proring $A$ and that $\cM\cong M^{\triangle}$
for a pro-Noetherian $A$-promodule $M=(M_{e})$.
We have to show that $M$ is Mittag-Leffler.
Put $\cM_{e} := \tilde M_{e}$, $\cM_{e}^{e'}:= \Im (\cM_{e'} \to \cM_{e})$, 
$e' \ge e$,
 and $\cM_{e}^{\infty} := \bigcap_{e' \ge e} \cM_{e}^{e'}$.
We now fix an index $e$.
For a point $x_{1} \in X$, since $\cM_{x_{1}}$ is Mittag-Leffler,
there exists $e_{1} \ge e$ such that
$(\cM^{e_{1}}_{e}/\cM^{\infty}_{e})_{x_{1}}=0$, that is, $x_{1} \notin \Supp (\cM^{e_{1}}_{e}/\cM^{\infty}_{e})$.
Here $\Supp (-)$ denotes the support of a sheaf.
If $\Supp (\cM^{e_{1}}_{e}/\cM^{\infty}_{e}) \ne \emptyset$, then 
we choose $x_{2} \in \Supp (\cM^{e_{1}}_{e}/\cM^{\infty}_{e}) $ 
and take $e_{2} \ge e_{1}$ such that $x_{2} \notin \Supp (\cM^{e_{2}}_{e}/\cM^{\infty}_{e}) $.
Then 
\[
\Supp (\cM^{e_{2}}_{e}/\cM^{\infty}_{e}) \subsetneq \Supp (\cM^{e_{1}}_{e}/\cM^{\infty}_{e}) .
\]
We can continue this procedure until we get empty $\Supp (\cM^{e_{i}}_{e}/\cM^{\infty}_{e}) $.
Since the underlying topological space of $X$ is Noetherian and
the $\Supp (\cM^{e_{i}}_{e}/\cM^{\infty}_{e}) $ are closed subsets,
for some $e _{i}$, $\Supp (\cM^{e_{i}}_{e}/\cM^{\infty}_{e}) =\emptyset$.
It proves the proposition.
\end{proof} 

\begin{cor}
Let $X$ be a locally ind-Noetherian formal scheme.
Then a sequence of locally pro-Noetherian $\cO_{X}$-promodules
\[
0 \to \cL \to \cM \to \cN \to 0
\]
is exact if and only if for every $x \in X$, the induced sequence of stalks
\[
0 \to \cL_{x} \to \cM_{x} \to \cN_{x} \to 0
\]
is exact.
\end{cor}

\begin{proof}
We can prove it like Proposition \ref{prop-exact-stalk}.
\end{proof}

\section{Formal schemes as functors and formal algebraic spaces}\label{sec-functors}

In this section, we see that a formal scheme can be considered
as a sheaf on the category of schemes.
Along this line, we also define a formal algebraic space. 

\subsection{Formal schemes as functors}

Let $\bF$ denotes the category of contravariant functors 
\[
\text{(Schemes)} \to \text{(Sets)},
\]
and $\bF_{\Zar}, \bF_{\et} \subseteq \bF$
the full subcategories of Zariski and \'etale sheaves respectively.
For a formal scheme $X$, we define $F_{X} \in \bF$ by
\[
F_{X}(Y) := \Hom_{\text{form.\ sch.}}(Y,X), 
\]
which is clearly a sheaf for both the Zariski and \'etale topologies.
As is well-known, the functor 
\begin{equation*}
F_{\bullet} : (\text{Schemes}) \to \bF, \, X \mapsto F_{X} 
\end{equation*}
is fully faithful.
So, by abuse of terminology, we say that a functor $F \in \bF$ 
is a scheme if $F \cong F_{X}$ 
for some scheme $X$.

\begin{defn}
A morphism $F \to G$ in $\bF$ is said to be {\em schematic}
if for every scheme $X$ and for every morphism $X \to G$,
the fiber product $F \times _{G} X$, which exists in $\bF$, is a scheme.
\end{defn}

We can generalize various properties of morphisms of schemes
to schematic morphisms: 

\begin{defn}\label{defn-property}
Let $\bP$ be a property of morphisms of schemes which is stable
under base changes.
We say that a schematic morphism $F \to G$ of $\bF$
has a property $\mathbf{P}$ if  for every scheme $X$ and for every morphism $X \to G$,
the projection $F \times _{G} X \to X$ has the property $\bP$.
\end{defn}

\begin{thm}
The functor 
\begin{equation*}
F_{\bullet} : (\textup{Formal schemes}) \to \bF, \, X \mapsto F_{X} 
\end{equation*}
 is fully faithful.
\end{thm}

\begin{proof}
Let $X$ and $Y$ be formal schemes.
First consider the case where $X = \varinjlim X_{d}$
for some admissible system $(X_{d})$ of schemes.
Given a morphism $F_{X} \to F_{Y}$, 
we have canonical elements $(X_{d} \to X) \in F_{X}(X_{d})$
and their images $(X_{d} \to Y) \in F_{Y}(X_{d})$,
which uniquely determine a morphism $X \to Y$.
It proves that the natural map
\[
\Hom (X,Y) \to \Hom (F_{X},F_{Y})
\]
is bijective.

Next consider the general case. Take an affine Zariski covering $X=\bigcup U_{i}$
and put $V:= \coprod U_{i}$ and $W := V \times_{X}V$.
Then there exist admissible systems $(V_{d})$ and $(W_{d})$
of schemes such that $V= \varinjlim V_{d}$ and $W = \varinjlim W_{d}$.
A morphism $F_{X} \to F_{Y}$ induces $F_{V} \to F_{Y}$ and 
$F_{W} \to F_{Y}$, and hence $V \to Y$ and $W \to Y$.
The last two morphisms is actually a gluing data of morphisms with respect to the Zariski topology,
so we  obtain a morphism $X \to Y$.
\end{proof}

\begin{rem}
Note that if $X$ is a formal scheme and $(X_{d})$ is an admissible 
system of schemes with $X=\varinjlim X_{d}$, 
then  $F_{X}$ is not isomorphic to 
the inductive limit $G$ of $F_{X_{d}}$'s in $\bF$, but isomorphic to its sheafification,
which is the inductive limit of $F_{X_{d}}$ in $\bF_{\Zar}$ and $\bF_{\et}$.
However $G$ is a sheaf on qsqc formal schemes. 
\end{rem}

\subsection{Formal algebraic spaces}

Again by abuse of terminology, we say that 
$F \in \bF$ is a formal scheme if it is isomorphic to $F_{X}$
with $X$ a formal scheme.
For a schematic morphism of formal schemes,
the immersions of Definitions \ref{defn-sub}
and \ref{defn-property} coincide thanks to the following:

\begin{lem}\label{lem-schematic-immersion}
Let $\phi :Y \to X$ be a schematic morphism of formal schemes.
Then $\phi$ is an immersion in the sense of 
Definitions \ref{defn-sub} if and only if for every morphism $W \to X$
with $W$ a scheme, so is the natural morphism $W \times_{X}Y \to Y$. 
\end{lem}

\begin{proof}
The ``only if'' follows from Proposition \ref{prop-immersion-basechange}.
We now prove  the ``if''. Considering the case $W =X_{\red}$,
we easily see that $\phi$ is a homeomorphism onto
a locally closed subset. 
For $y \in Y$, if we write $\cO_{X,\phi(y)}=(A_{d})$ where the $A_{d}$
are rings, then by assumption, $B_{d}:=A_{d} \otimes_{\cO_{X,\phi(y)}}\cO_{Y,y}$ is
a ring up to isomorphisms, and the natural morphism $f_{d}:A_{d} \to B_{d}$
is surjective. 
Then the natural morphism $\cO_{X,\phi(y)}\to \cO_{Y,y}$ is equal to 
$\prolim f_{d}$ and hence an epimorphism, which proves the lemma.
\end{proof}

\begin{lem}\label{lem-algsp}
\begin{enumerate}
\item
For a formal scheme $X$, the diagonal morphism $X \to X \times X$
is schematic and an immersion.
\item 
Let $Y \to X$ be a morphism of $\bF$ which is schematic and an immersion.
If $X$ is a formal scheme, then so is $Y$.
\item
Let $F \in \bF$ such that the diagonal morphism $F \to F \times F$
is schematic and an immersion,
and let $U \to F$ and $V \to F$ be 
morphisms of $\bF$ with $U$ and $V$ formal schemes.
Then $U \times _{F} V$ is a formal scheme.
\end{enumerate}
\end{lem}

\begin{proof}
(1). We may suppose that $X$ is affine and 
there exists an admissible system of affine schemes $(X_{d})$ with 
$X= \varinjlim X_{d}$. Let $V$ be a quasicompact scheme and $V \to X \times X$
an arbitrary morphism. Then  
\[
 X \times_{X \times X} V  \cong \varinjlim X_{d} \times_{X_{d} \times X_{d}} V,
\]
where $d$ runs over those indices such that
$V \to X \times X$ factors through $X_{d} \times X_{d}$.
Since for $d' \ge d$, the natural diagram
\[
\xymatrix{
 X_{d} \ar[r] \ar[d] & X_{d} \times X_{d} \ar[d] \\
 X_{d'} \ar[r] & X_{d'} \times X_{d'} 
}
\]
is cartesian, the inductive system $(X_{d} \times_{X_{d} \times X_{d}} V)$
is constant. It proves the proposition.

(2). First suppose that there exists an admissible system of schemes
$(X_{d})$ with $X=\varinjlim X_{d}$. Then for each $d$, $Y_{d}:=Y \times_{X} X_{d}$ is
a subscheme of $X_{d}$, and $(Y_{d})$ is an admissible system of schemes.  
Then $Y$ is the inductive limit of $(Y_{d})$ say in $\bF_{\Zar}$,
which is a formal subscheme of $X$ and hence a formal scheme.
(I do not know if for any admissible system $(X_{d})$ of schemes,
its limit in $\bF_{\Zar}$ is a formal scheme. 
So I have to add the condition of being an immersion.
The problem is that I do not know
if for a bonding morphism $X_{d} \to X_{d'}$, the image of an affine open is affine,
and if the limit of $(X_{d})$ as a proringed space is covered by affine formal schemes.)

In the general case, we take a Zariski covering $U \to X$ such that
$U=\varinjlim U_{d}$ for some admissible system $(U_{d})$. 
If we put $V := U \times_{X} U$,then
there exists an admissible system $(V_{d})$ with $V =\varinjlim V_{d}$.
Then $Y_{U}:= Y \times _{X}U$ and $Y_{V}:= Y \times _{X} V$ are formal schemes. 
From the gluing data $(Y_{U}, Y_{V})$ of formal schemes, 
we obtain a formal scheme, which is nothing but $Y$. 

(3). The natural morphism $U \times _{F} V \to U \times V $ is
a base change of $F \to F \times F$, so schematic and an immersion. 
From (2), $U \times_{F} V$ is a formal scheme.
\end{proof}

As the definition of formal algebraic space, we adopt the following one:

\begin{defn}\label{def-as-functor}
An \'etale sheaf $X \in \bF_{\et}$ is called a {\em formal algebraic space}
if the diagonal morphism $X \to X \times X$ is schematic and an immersion,
and there exists a schematic \'etale morphism $U \to X$
with $U$ a formal scheme.
\end{defn}

From Lemma \ref{lem-algsp} (1), a formal scheme is a formal algebraic space.
For a formal algebraic space $X$ and $U \to X$ as in the definition,
from Lemma \ref{lem-algsp} (3), $R:= U \times_{X} U$ is a formal scheme.
The natural morphism $R \to U \times U$ is schematic and an immersion,
because it is a base change of $X \to X \times X$. 
The two projections $R \rightrightarrows U$ are schematic and \'etale. 
Thus $R$ is an \'etale equivalence relation on $U$,
and $X$ is the quotient $R/U$ in $\bF_{\et}$.
We have also the \'etale equivalence  relation $R_{\red}$ on the reduced scheme
$U_{\red}$ and obtain a reduced algebraic space $X_{\red}:=U_{\red}/R_{\red}$,
which is a generalization of the one defined for a formal scheme. 

Conversely given an equivalence relation $R \to U \times U$ in $\bF_{\et}$
such that $R$ and $U$ are formal schemes, $R \to U \times U$ is schematic
and an immersion, and $R \rightrightarrows U$ are schematic and \'etale,
then the quotient $R/U$ is a formal algebraic space.

\begin{defn}
For a formal algebraic space $X$, 
we define the {\em \'etale site}, $X_{\et}$, as  the category of  formal algebraic spaces $Y$ 
which are schematic and \'etale over $X$
with the obvious notion of covering.
\end{defn}

We define the {\em structure sheaf} $\cO^{\et}_{X}$ on $X_{\et}$ as follows,
which makes $X_{\et}$ a ``proringed site'': 
For $(U \to X) \in X_{\et}$ with $U$ a qsqc 
formal {\em scheme}, $\cO^{\et}_{X}(U) := \cO_{U}(U)$.
From the following lemma, 
this defines a sheaf on the  basis consisting of all such $U$'s.

\begin{lem}
Let $V \to U$ be an \'etale covering of formal schemes and $W:= V \times_{U} V$.
Then the sequence
\[
0 \to \cO_{U} (U) \to \cO_{V}(V) \rightrightarrows \cO_{W}(W)
\]
is exact. 
\end{lem}

\begin{proof}
Take an admissible system $(U_{d})$ of schemes with $U=\varinjlim U_{d}$,
and put $V_{d} := U_{d} \times _{U} V$ and $W_{d}:=U_{d} \times_{U}W$.
Then $(V_{d})$ and $(W_{d})$ are admissible systems of schemes such that
 $V = \varinjlim V_{d}$ and $W = \varinjlim W_{d}$.
For each $d$, we have the exact sequence of rings
\[
 0 \to \cO_{U_{d}} (U_{d}) \to \cO_{V_{d}} (V_{d}) \rightrightarrows \cO_{W_{d}} (W_{d}).
\]
Taking the projective limit, we obtain the exact sequence in the lemma.
\end{proof}

Then we can uniquely
extend  $\cO^{\et}_{X}$ to  $X_{\et}$ as a sheaf on qsqc objects.
Here the qsqc object is defined just like the qsqc topological space. 
Then we easily see that for a formal scheme $ \in X_{\et}$,
being qsqc as an object of $X_{\et}$ is equivalent
to being qsqc as a topological space.

\begin{defn}
An  $\cO^{\et}_{X}$-promodule $\cM$ is said to be {\em semicoherent}
if for some \'etale covering $V \to X$ with $V$ a formal scheme,
the restriction  of $\cM$ to the Zariski site of $V$
is a semicoherent $\cO_{V}$-promodule.
\end{defn}

If $X$ is a formal scheme 
and $\cM$ is a semicoherent $\cO_{X}$-promodule,
then we define a semicoherent $\cO_{X}^{\et}$-promodule 
$\cM^{\et}$ so that for an \'etale morphism $\phi:U \to X$ with $U$ a formal scheme,
$\cM^{\et} (U) =( \phi^{*} \cM)(U)$.
Restricting $\cM^{\et}$ to the Zariski site of $X$, we can recover $\cM$.

\begin{lem}
If $X=\Spf A$ and if $\cM$ is a semicoherent $\cO_{X}^{\et}$-promodule,
then  $\cM \cong (\cM(X)^{\triangle})^{\et}$.
\end{lem}

\begin{proof}
It is proved in the same manner as Proposition \ref{prop-affine-semicoh}.
\end{proof}

As a consequence, for a formal scheme $X$, we obtain the equivalence
\[
\text{(Semicoherent $\cO_{X}$-promodules)} \cong \text{(Semicoherent $\cO_{X}^{\et}$-promodules)} .
\]

\begin{defn}
A morphism $\phi : Y \to X$ of formal algebraic spaces is {\em formally schematic}
if for every morphism $W \to X$ with $W$ a formal scheme, 
$Y \times _{X} W$ is also a formal scheme. 
A formally schematic morphism $\phi : Y \to X$ of formal algebraic spaces
is a {\em (resp. closed, open) immersion}
if for every morphism $W \to X$ with $W$ a formal scheme, 
the projection $Y \times _{X} W \to W$ is
a (resp. closed, open) immersion.
For a formal algebraic space $X$,
a {\em (resp.\ closed, open) formal algebraic subspace} of $X$ 
is an isomorphism classes of (resp. closed, open) immersions $Y \to X$.
\end{defn}

We can now translate all the results in \S \ref{sec-semicoh} to 
formal algebraic spaces. 
Note that  in Proposition \ref{prop-ext}, we have to replace
the open subset $U \subseteq X$  not with an \'etale morphism but with
  an open immersion $U \to X$ of formal algebraic spaces.

\section{Mild and gentle formal schemes}\label{sec-mild-gentle}

This section establishes the relation between
prorings (resp.\ promodules) and complete rings (resp.\ complete modules).
Also we introduce two classes of formal schemes
which are well-behaved when completing
the structure sheaf.

\subsection{Complete rings and complete modules}

A {\em linearly topologized ring} is a topological ring 
which admits a basis of open neighborhoods of $0$
consisting of ideals.
Such a basis is called a {\em basis of open ideals}.
Conversely if a ring $A$ is given a collection $(I_{d})$ of ideals 
which is directed with respect to the preorder $I_{d} \le I_{d'} \Leftrightarrow I_{d} \supseteq I_{d'}$,
then there exists a unique topology on $A$ for which $A$ is linearly topologized
and $(I_{d})$ is a basis of open ideals.
We call this topology the $(I_{d})$-{\em topology}. 
A linearly topologized ring is said to be {\em gentle}
if it has a {\em countable} basis of open ideals.
If $A$ is a linearly topologized ring and
$(I_{d})$ is a basis of open ideals,
then the {\em completion} $\hat A$ of $A$ is defined to be the projective limit
$\varprojlim A/I_{d} $ which is linearly topologized 
so that the kernels of $\hat A \to A/I_{d}$ form 
a basis of open ideals.
A {\em complete ring} is a linearly topologized ring $A$ such that 
the natural map $A \to \hat A$ is bijective.
As a special case, every ring with the discrete topology
is a complete linearly topologized ring.
We define a {\em morphism} of complete rings as
 a homomorphism of rings which is continuous.

Let $A$ be a linearly topologized ring
and $M$ an abelian topological group  endowed with an $A$-module structure.
We call $M$  a {\em linearly topologized $A$-module}
if it has a basis $(M_{e})$ of open neighborhoods of $0$
consisting of $A$-submodules $M_{e} \subseteq M$
and if for every open neighborhood $0 \in V \subseteq M$,
there exists an open ideal $I \subseteq A$ with $IM \subseteq V$. 
Such a basis is called a {\em basis of open submodules}.
A linearly topologized $A$-module is said to be {\em gentle}
if it admits a countable basis of open submodules.
If $M$ is a linearly topologized $A$-module, with the above notation, 
the completion $\hat M$ is defined to be $\varprojlim M/M_{e}$,
which is a linearly topologized $\hat A$-module as well as a linearly topologized
$A$-module.
A {\em complete $A$-module} is a linearly topologized 
$A$-module $M$ with  $M = \hat M$.
We define a {\em morphism} of complete $A$-modules
as a homomorphism of $A$-modules which is continuous.
 
\subsection{Pro vs.\ complete}

Let  $A=(A_{d})$ be a proring.
We give  to the projective limit, $\hat A:= \varprojlim A_{d}$,
in the category of rings the topology such that
 $\hat A$ is linearly topologized
and the kernels $\Ker (\hat A \to A_{d})$ form a basis of neighborhoods of $0$.
This defines a functor
\[
\wedge : \text{(Prorings)} \to \text{(Complete rings)},\ A \mapsto \hat A.
\]
Conversely if $B$ is a complete ring and $(I_{d})$ is a basis  of
open ideals, then $(B/I_{d})$ is an epi proring.
The isomorphism class of $(B/I_{d})$ is independent of the choice of $(I_{d})$.
To kill the ambiguity, taking the set $(I_{d})$ of all open ideals, we
put $\check{B}:=(B/I_{d}) $, which defines a functor
\[
\vee : \text{(Complete rings)} \to \text{(Epi prorings)},\ B \mapsto \check B.
\]
This is a fully faithful embedding and  
the composite functor $\wedge \circ \vee$ is isomorphic to the identity.

\begin{defn}
A proring $A$ is said to be {\em mild} if $A \cong \check B$ for some
complete ring $B$.
\end{defn}

A mild proring  is by definition isomorphic to an epi proring,
and hence Mittag-Leffler. But there exists an epi proring which is not mild (see 
Example \ref{expl-nonmild}).\footnote{It was a referee who let me know the existence of
non-mild and epi proring.}

Similarly, for a proring $A$,  we  have
\[
\wedge  : \text{($A$-promodules)} \to \text{(Complete $\hat A$-modules)},\ M \mapsto \hat M
\]
and for a complete ring $B$,
\[
\vee:\text{(Complete $B$-modules)} \to \text{(Epi $\check B$-promodules)},\ N \mapsto \check N.
\]
which is fully faithful. 

\begin{defn}
Let $A$  be a mild proring.
An $A$-promodule $M$ is said to be {\em mild} 
if $M \cong \check N$ for some complete $\hat A$-module $N$.
\end{defn}

For a mild proring $A$, we have
\[
\text{(Complete $\hat A$-modules)} \cong \text{(Mild $A$-promodules)}.
\]

\begin{prop}\label{prop-mild-criterion}
\begin{enumerate}
\item
An epi proring $A=(A_{d})$ is mild if and only if the natural maps $\hat A \to A_{d}$
are surjective.
\item Let $A$ be a mild proring. Then an epi $A$-promodule $M=(M_{e})$
is mild if and only if the natural maps $\hat M \to M_{e}$ are surjective.
\end{enumerate}
\end{prop}

\begin{proof}
Since the proofs of (1) and (2) are parallel, we only prove (1).
We first prove the ``if'' part. Put  $I_{d} := \Ker (\hat A \to A_{d})$.
Then by definition, $(I_{d})$ is a basis of open ideals and 
\[
A=(A_{d})=(\hat A /I_{d}) \cong \check {\hat{A}}.
\]
So $A$ is mild.

Next we prove the ``only if'' part. Let $B$ be a complete ring such that $A \cong \check B$.
We write $\check B = (B_{e})$. By construction, we have the natural surjections
$B \to B_{e}$.
From Lemma \ref{lem-surjective}, for each $d$, 
there exists a surjection $B_{e} \to A_{d}$ which represents an isomorphism $\check B \to A$.
So the natural morphism
\[
\hat A \cong B \to B _{e} \to A_{d}
\]
is surjective.
\end{proof}

\begin{defn}
A pro-object of any category is said to be {\em gentle}
if it is isomorphic in the procategory to a pro-object indexed by 
a countable directed set, or equivalently to one indexed by $\NN$.
\end{defn}

From Proposition \ref{prop-mild-criterion} and 
the construction of the projective limit, 
every gentle and Mittag-Leffler proring is mild, and every gentle
and Mittag-Leffler $A$-promodule
for a mild proring $A$ is mild.
Hence
\[
\text{(Gentle complete rings)} \cong \text{(Gentle and Mittag-Leffler prorings)},
\]
and if $B$ is a complete ring and $A$ is the associated mild proring, 
\[
\text{(Gentle complete $B$-modules)} \cong \text{(Gentle and Mittag-Leffler $A$-promodules)}.
\]

\begin{expl}\label{expl-nonmild}
There exists a directed projective system $(S_{d})$ of sets
with surjective bonding maps and $\varprojlim S_{d} = \emptyset$ (see \cite{Henkin, Higman-Stone, Waterhouse}).
From such a system, imitating a construction of Higman and Stone  \cite{Higman-Stone},
we can construct an admissible proring which is not mild:
For $d' \ge d$,  we have a natural surjective 
homomorphism of polynomial rings (possibly with infinite variables)
\begin{equation}\label{poly-hom}
k[x_{s};s \in S_{d'}] \to k[x_{s};s\in S_{d}] .
\end{equation}
Here $k$ is a filed. 
Consider the quotient ring
\[
R_{d} :=k[x_{s};s \in S_{d'}]/ (x_{s}; s \in S_{d})^{2} ,
\]
 which is, as a vector space, isomorphic to 
 \[
 k  \cdot 1 \oplus \bigoplus_{s \in S_{d}} k \cdot x_{s} .
 \]
The homomorphism  \eqref{poly-hom} induces a surjective homomorphism $R_{d'} \to R_{d}$
and yields an epi admissible proring $R:=(R_{d})$.
We easily see that $\hat R =k$. So $R$ is not mild.
\end{expl}

\begin{cor}\label{cor-quot-mild}
\begin{enumerate}
\item Let $A \to B$ be a morphism of prorings which is an epimorphism
of $A$-promodules. If $A$ is mild, then so is $B$.
\item Let $A$ be a mild proring and $M \to N$ an epimorphism
of $A$-promodules. If $M$ is mild, then so is $N$.
\end{enumerate}
\end{cor}

\begin{proof}
Again we prove only (1). Being mild, $A$ is Mittag-Leffler and so is $B$.
Hence we may suppose that $A=(A_{d})$ and 
$B=(B_{e})$ are epi. For each $e$, take a homomorphism 
$A_{d} \to B_{e}$ representing the given $A \to B$.
From Lemma \ref{lem-surjective}, it is surjective. Since $A$ is mild,
from Proposition \ref{prop-mild-criterion}, 
the natural map $\hat A \to A_{d} \to B_{e}$ is surjective.
Since $\hat A \to B_{e}$ factors as $\hat A \to \hat B \to B_{e}$,
the natural map $\hat B \to B_{e}$ is surjective, which proves the assertion.
\end{proof}

\begin{cor}\label{cor-tensor-mild}
\begin{enumerate}
\item Let $A \to B$ and $A \to C$ be morphisms of prorings.
If $B$ and $C$ are mild, then so is $B \otimes_{A} C$.
If $A$, $B$ and $C$ are gentle, then so is $B \otimes_{A} C$.
\item Let $A$ be a mild (resp.\ gentle) proring, and $M$ and $N$ mild (resp.\ gentle) $A$-promodules.
Then $M \otimes_{A} N$ is (resp.\ gentle) mild.
\end{enumerate}
\end{cor}

\begin{proof}
Again we prove only (1). We first consider the case where $B$ and $C$
are mild. Then we may suppose that $B$ and $C$ are epi
and that $A$, $B$ and $C$ has a same index set and
the given $A \to B$ and $A \to C$ are represented by level morphisms
$(A_{d} \to B_{d})$ and $(A_{d} \to C_{d})$ respectively.
Choose an index $d_{0}$ and arbitrary elements $b \in B_{d_{0}}$
and $c  \in C_{d_{0}}$. 
From Proposition \ref{prop-mild-criterion}, there exists $(b_{d}) \in \hat B=\varprojlim B_{d}$
and $(c_{d}) \in \hat C =\varprojlim C_{d}$ such that $b_{d_{0}}=b$
and $c_{d_{0}}=c$. Then 
\[
 (b_{d} \otimes c_{d}) \in \widehat{ B \otimes _{A} C} = \varprojlim B_{d} \otimes_{A_{d}} C_{d}.
\]
Hence $b \otimes c \in \Im (\widehat{ B \otimes _{A} C}  \to B_{d} \otimes _{A_{d}} C_{d} )$.
It follows that $\widehat{ B \otimes _{A} C}  \to B_{d} \otimes _{A_{d}} C_{d} $ is
surjective. From Proposition \ref{prop-mild-criterion}, this means that $B \otimes _{A} C$ is mild.

Next we consider the case where $A$, $B$ and $C$ are gentle.
Then we may suppose that $A=(A_{d})$, $B=(B_{e})$ and $C=(C_{f})$
are indexed by  countable sets. Then $B \otimes _{A}C =
(B_{e} \otimes_{A_{d}}  C_{f})$ is also indexed by a countable set, so gentle. 
\end{proof}

\begin{lem}\label{lem-epi}
Let $A$ be a mild proring, $\phi:M \to N$ a morphism of mild $A$-promodules
and $\hat \phi:\hat M \to \hat N$ the corresponding morphism of 
complete $\hat A$-modules. Then the following are equivalent:
\begin{enumerate}
\item $\phi$ is an epimorphism in the category of $A$-promodules.
\item $\phi$ is an epimorphism in the category of mild $A$-promodules.
\item $\hat \phi$ is an epimorphism in the category of complete
$\hat A$-modules.
\end{enumerate}
\end{lem}

\begin{proof}
We obviously have $ (1)\Rightarrow  (2) \Leftrightarrow (3)$.
It remains to show $(2) \Rightarrow (1)$.
Suppose $(2)$ and write $M=(M_{d})$ and $N=(N_{e})$,
which may be supposed to be epi.
For every $e$, the composite morphism $\phi_{e}:M \to N \to N_{e}$ is an epimorphism
in the category of mild $A$-promodules.
Therefore every morphism $\phi^{d}_{e}:M_{d} \to N_{e}$
which represents $\phi_{e}$  must be surjective.
Hence $\phi$ is an epimorphism also in the category of $A$-promodules.
\end{proof}

\subsection{Admissible  rings}\label{subsec-admissible}

An open ideal $I $ of a topological ring  is called an {\em ideal of definition}
if every element $a \in I$ is topologically nilpotent (that is, $a^n \to 0$, as $n \to \infty$).\footnote{This 
definition is due to McQuillan \cite{McQuillan}.
The one in \cite{EGA} is more restrictive:
In \textit{op.\ cit.}, an ideal $I$ is an {\em ideal of definition} if
for every open neighborhood  $V$ of $0$, there exists $n \in \NN$ with $I^n \subseteq V$.}\label{footnote-ideal-of-def} 
A linearly topologized ring is called an {\em admissible ring}
if  it is  complete and admits an ideal of definition. 
Especially every discrete ring is admissible.
Every admissible ring has the largest ideal of definition, which is 
the ideal of
all the topologically nilpotent elements.\footnote{This fails if we adopt the definition in \cite{EGA}.}

A collection of ideals of definition in a topological ring $A$
is called a {\em basis of ideals of definition} if it is a basis 
of open ideals. 
If $A$ is an admissible ring, $J \subseteq A$ is an ideal of definition
and $(I_{d})$ is a basis of open ideals, then   $( I_d \cap J ) $
 is a basis of ideals of definition.
Thus every admissible ring admits a basis of ideals of definition.
In particular, the collection of all ideals of definition is a basis of 
ideals of definition.

If $A=(A_{d})$ is an admissible proring, then $\hat A$ is an admissible ring.
Conversely if $B$ is an admissible ring and $(I_{d})$ is the set of all
ideals of definition, then $(B/I_{d})$ is a mild admissible proring.
Hence 
\[
\text{(Mild admissible prorings)} \cong \text{(Admissible rings)}.
\]
This induces the equivalence of subcategories,
\[
\text{(Gentle admissible prorings)} \cong \text{(Gentle admissible rings)}.
\]

\subsection{Mild and gentle formal schemes}

\begin{defn}
A formal scheme is said to be {\em mildly (resp.\ gently) affine}
if it is  isomorphic to $\Spf A$ with $A$ mild (resp.\ gentle).
A formal scheme $X$ is said to be {\em mild (resp.\ gentle)} if every point $x \in X$ admits
a mildly (resp.\ gently) affine neighborhood.
For a mild  (resp.\ gentle) formal scheme $X$, a semicoherent $\cO_{X}$-promodule
$\cM$ is said to be {\em mild (resp.\ gentle)} if every point $x \in X$
admits a mildly  (resp.\ gently) affine neighborhood $U=\Spf A \subseteq X$ such that
$\cM(U)$ is a mild  (resp.\ gentle) $A$-promodule.
\end{defn}

By definition, every gentle formal scheme is mild.

\begin{prop}\label{prop-gentle-affine}
Suppose that an affine formal scheme $X=\Spf A$ is gentle.
Then $A$ is gentle. 
\end{prop}

\begin{proof}
Take a finite gently affine covering  $X = \bigcup \Spf A_{i}$.
Put $A_{ij} := \Spf A_{ij}$, $B:=\prod A_{i}$ and $C:= \prod C_{ij}$.
We have an exact sequence
\[
 0 \to A \to B \rightrightarrows C .
\]
Since $B$ and $C$ are gentle, the morphisms $B \rightrightarrows C$
are represented by  level morphisms $(B_{i} \rightrightarrows C_{i})_{i \in \NN}$.
Then $A \cong (\Ker (B_{i} \rightrightarrows C_{i}))_{i \in \NN}$, and hence $A$ is gentle.
\end{proof}

\begin{prop}
\begin{enumerate}
\item
Every distinguished open subscheme of a mildly (resp.\ gently)
affine formal scheme is mildly (resp.\ gently) affine.
\item Let $X$ be a mild (resp.\ gentle) formal scheme.
Then for every $x \in X$, $\cO_{X,x}$ is mild (resp.\ gentle).
\end{enumerate}
\end{prop}

\begin{proof}
(1) Let $A:=(A_{d})$ be a mild admissible proring.
Set $X:=\Spf A$ and $X_{d}:=\Spec A_{d}$.
The natural map $\hat A \to A_{\red}$ is surjective.
Take an arbitrary $f \in A_{\red}$.
Then $D(f)= \Spf A_{f}$, where $A_{f}=(A_{d,f})$.
For any lift $f' \in \hat A$ of $f$, 
the natural map $(\hat A) _{f'} \to A_{d, f}$ is surjective
and
factors as $(\hat A)_{f'} \to \widehat{ A_{f} } \to A_{d, f}$.
So $\widehat{ A_{f} } \to A_{d, f}$ is surjective and $A_{f}$ is mild. It shows the assertion
for the mildly affine formal scheme. 
The assertion for the gently affine formal scheme is trivial.

(2) The proof is parallel to the one of (1).
\end{proof}

\begin{prop}
Let $X:=\Spf A$ be a gentle affine formal scheme and $\cM$ a gentle semicoherent
$\cO_{X}$-promodules. Then $\cM(X)$ is a gentle $A$-promodule.
Hence  we have the equivalence of abelian categories
\[
 \text{(Gentle $A$-promodules)} \cong \text{(Gentle semicoherent $\cO_{X}$-promodules)}
\]
\end{prop}

\begin{proof}
It can be proved in the same way as Proposition \ref{prop-gentle-affine}.
\end{proof}

\begin{prop}
Let $X$ be a mild formal scheme, $\cM$ a mild semicoherent 
$\cO_{X}$-promodule and $\cN$ a quotient semicoherent 
$\cO_{X}$-promodule of $\cM$. Then $\cN$ is mild.
\end{prop}

\begin{proof}
It is a direct consequence of Corollary \ref{cor-quot-mild}.
\end{proof}

\begin{cor}
Every formal subscheme of a mild formal scheme is mild.
\end{cor}

\begin{proof}
It is clear that every {\em open} formal subscheme of a mild formal scheme
is mild. So it is enough to show that every {\em closed} formal subscheme
of a mild formal scheme is mild, which follows from the preceding proposition.
\end{proof}

\begin{prop}
\begin{enumerate}
\item
Let $Y \to X$ and $Z \to X$ be  morphisms of formal schemes.
If  $Y$ and $Z$ are mild, then so is $Y \times_{X} Z$.
If $X$, $Y$ and $Z$ are gentle, then so is $Y \times _{X} Z$.
In particular, the category of mild (resp.\ gentle) formal schemes is closed under fiber
products.
\item  Let $\phi: Y \to X$ be a morphism of mild (resp.\ gentle)
formal schemes
and $\cM$ a mild (resp.\ gentle) semicoherent $\cO_{X}$-promodule.
Then $\phi^{*} \cM$ is mild (resp.\ gentle).
\end{enumerate}
\end{prop}

\begin{proof}
The assertions follow from Corollary \ref{cor-tensor-mild}.
\end{proof}

\subsection{Complete sheaves}

\begin{defn}
Let $X$ be a formal scheme. We define the {\em complete structure
sheaf} $\hat \cO_{X}$, which is a sheaf of complete rings,
 by $\hat \cO_{X} (U):= \widehat{ \cO_{X}(U)} $
for qsqc $U \subseteq X$,
and the {\em complete stalk} $\hat \cO_{X,x}$ to
be the complete ring  corresponding to the proring $\cO_{X,x}$.
\end{defn}

Since the projective limit is left exact, 
the complete structure sheaf is indeed a sheaf of complete  rings.
If $X$ is mild, then every stalk $\cO_{X,x}$ is mild
and $\hat \cO_{X,x}$ is a local ring.

For mild formal schemes $X$ and $Y$,
a morphism $\phi:Y \to X$ gives the data of a continuous map $\phi:Y \to X$
denoted by the same symbol and
 a $\phi$-morphism $\hat \cO_{X} \to \hat \cO_{X} $
which induces local homomorphisms
$\hat \cO_{X,\phi(y)} \to \hat \cO_{Y,y}$, and vice versa.

\begin{conv}
Let  $A$ be a mild admissible proring and $B$ the corresponding 
admissible ring. 
Then by abuse of notation, we also write $\Spf B$ for $\Spf A$.
\end{conv}

The underlying topological space of $\Spf B$ is identified with the set of open prime ideals of $B$.

\begin{defn}
Let $X$ be a formal scheme. A {\em complete $\hat \cO_{X}$-module}
is a sheaf $\cM$ of complete abelian groups such that for each qsqc $U \subseteq X$,
$\cM(U)$ is given a complete $\hat \cO_{X}(U)$-module structure
in a compatible way.
\end{defn}

If $X$ is a formal scheme and $\cM$ is an $\cO_{X}$-promodule,
then putting $\hat \cM(U) := \widehat{\cM(U)}$,
we obtain a complete $\hat \cO_{X}$-module $\hat \cM$.

\begin{defn}
Let $X$ be a mild formal scheme.
A complete $\hat \cO_{X}$-module $\cN$ is said to be {\em semicoherent}
if $\cN \cong \hat \cM$ for some mild semicoherent $\cO_{X}$-promodule $\cM$.
If in addition $X$ is gentle and $\cM$ is gentle and Mittag-Leffler, then we say that $\cN$ is {\em gentle}.
\end{defn}

From the definition, we obtain:

\begin{prop}
\begin{enumerate}
\item 
For a mild formal scheme $X$,
\[
\textup{(Mild semicoherent $\cO_{X}$-promodules)} \cong 
\textup{(Semicoherent complete $\hat \cO_{X}$-modules)}.
\]
\item 
For a gentle formal scheme $X$,
\begin{multline*}
\textup{(Gentle Mittag-Leffler semicoherent $\cO_{X}$-promodules)} \\ 
\cong 
\textup{(Gentle semicoherent complete $\hat \cO_{X}$-modules)}.
\end{multline*}
\item 
For a gentle admissible ring $A$, if we put $X:=\Spf A$,
\[
\textup{(Gentle semicoherent complete $\hat \cO_{X}$-modules)} \\
\cong \textup{(Gentle complete $A$-modules)}.
\]
\end{enumerate}
\end{prop}

\section{Local properties of mild formal schemes}\label{sec-local-properties}

We will study local properties of mild formal schemes
and their formal subschemes in terms of complete rings.

\subsection{Adic and Noetherian admissible rings}

\begin{defn}
A complete ring $A$ is said to be {\em adic} if
there exists an ideal $I \subseteq A$ such that 
$I^{n}$, $n \in \NN$, form a basis of open ideals.
\end{defn}

By definition, an adic complete ring is admissible and gentle.

\begin{defn}
An admissible ring is said to be {\em Noetherian} if it is Noetherian 
as a ring.
An admissible ring $A$ is said to be {\em pro-Noetherian}
if its corresponding admissible proring is pro-Noetherian, or equivalently
if for every open ideal $I \subseteq A$, $A/I$ is Noetherian.
\end{defn}

It is obvious that  every Noetherian admissible ring is  pro-Noetherian. 

\begin{expl}\label{ex-kxy-nonadic}
Let $A$ be the ring $k[[x,y]]$ of formal power series in
variables $x$ and $y$
over a field $k$ with the $((xy^{n}))_{n \in \NN}$-topology.
Then $A$ is a Noetherian admissible ring which is not adic.
\end{expl}

\begin{defn}
A mild formal scheme $X$ is said to be {\em locally Noetherian} 
(resp.\  {\em adic})
if  every $x \in X$ has an affine neighborhood $x \in \Spf A \subseteq X$
with $\hat A$ Noetherian and adic (resp.\  adic).
\end{defn}

The definition is due to \cite{EGA}.
We note that if $A$ is Noetherian but not adic, then $\Spf A$ is not locally Noetherian.

\begin{prop}
Let $A$ be an admissible ring such that $\Spf A$ is locally Noetherian.
Then $ A$ is Noetherian and adic.
\end{prop}

\begin{proof}
\cite[Chap.\ I, Cor.\ 10.6.5]{EGA}. 
\end{proof}

\begin{prop}\label{prop-pro-Noetherian-admissible}
Let $A$ be a pro-Noetherian admissible ring and $I \subseteq A$ an ideal of definition.
Then for any neighborhood $V$ of $0$, there exists $n \in \NN$ with $I^n \subseteq V$.
(Namely, in the sense of \cite{EGA}, $I$ is an ideal of definition 
and $A$ is admissible. See Footnote \ref{footnote-ideal-of-def}, page \pageref{footnote-ideal-of-def})
\end{prop}

\begin{proof}
Since $A$ is linearly topologized, we may suppose that $V$ is an ideal.
Then $ A/V $ is Noetherian. Therefore $I(A/V) $ is finitely generated.
Since every element of $I(A/V) $ is nilpotent, so is $I(A/V) $. This means
that for some $n$, $I^n \subseteq V$.
\end{proof}

From this proposition, our locally ind-Noetherian mild formal scheme
determines a formal scheme in the sense of \cite{EGA}.

\begin{prop}\label{prop-adic+pro-Noetherian-Noetherian}
Every pro-Noetherian adic ring $A$ is Noetherian.
Furthermore for every ideal $I$ of definition in $A$,
the topology on $A$ is identical to the $I$-adic topology.
\end{prop}

\begin{proof}
Let  $I \subseteq A$ be an ideal such that  $(I^n)_{n \in \NN}$
is a basis of ideals of definition. 
By  definition, $A/I$ and $A/I^2$ are  Noetherian. 
Consequently $I/I^2$ is finitely generated, and from \cite[0, Cor.\ 7.2.6]{EGA},
$A$ is Noetherian. 

Let $J$ be an arbitrary ideal of definition.
Then for some $m \in \NN$, $I ^m \subseteq J$. 
Hence for every $n \in \NN$, $I ^{mn} \subseteq J ^n$, and so $J^n$ is open.
Conversely, since $J$ is finitely generated, for every $n \in \NN$, there exists
$m \in \NN$ with $J^m \subseteq I^n$. This proves the lemma.
\end{proof}

\begin{cor}\label{cor-ind-adic-Noet}
Every mild, locally ind-Noetherian and adic formal scheme is locally Noetherian.
\end{cor}

\begin{proof}
Obvious.
\end{proof}

\subsection{Strict formal  subschemes}\label{subsec-what-subschemes}

Let $A$ be an admissible ring. 
The category of complete $A$-modules is an additive category with kernels and cokernels.
For a morphism $\phi :N \to M$ of complete $A$-modules,
we have the kernel $K \subseteq N$ in the category of $A$-modules.
Since $K \subseteq N$ is  closed,
it is a complete $A$-module with respect to the subspace topology,
and the kernel of $\phi$ also in the category of complete
$A$-modules. 

For every closed $A$-submodule $P $
of a complete $A$-module $M$, we can construct 
the {\em complete quotient}, denoted $M \sslash P$,
as the completion of the usual quotient $M / P$
which is endowed with the quotient topology.

If $M$ is gentle,
then $M/P$ 
is complete, so $M\sslash P=M/P$ (for example, see \cite[Th.\ 8.1]{Matsumura}).
Now the cokernel of a morphism $\phi:N \to M$ of 
complete $A$-modules is 
$M\sslash \overline{\Im \phi}$.
Here $\overline{\Im \phi}$ is the closure of $\Im \phi \subseteq M$.

\begin{defn}
Let $M \to N$ be an epimorphism of complete $A$-modules.
We say that $N$ is  a {\em normal quotient} of $M$ if 
$N \cong M \sslash \Ker (\phi)$.
\end{defn}

Let $M$ be a complete $A$-module, $(M_{d})$ a basis of open 
$A$-modules and $P \subseteq M$ a closed $A$-submodule.
Then we have the exact sequences of discrete $A$-modules
\[
0 \to P/(P \cap M_{d}) \to M/ M_{d} \to M/ (P+M_{d}) \to 0
\]
and the exact sequence of $\check A$-promodules
\[
0 \to (P/(P \cap M_{d})) \to (M/ M_{d}) \to (M/ (P+M_{d})) \to 0.
\]
Applying the completion functor
to the last exact sequence, we obtain the exact sequence of 
complete $A$-modules
\[
0 \to P \to M \to M \sslash P \to 0 .
\]

\begin{lem}\label{lem-normal-Mittag}
Let $A$ be an admissible ring and $M \to N$ an epimorphism of complete $A$-modules. Then $N$ is a normal quotient
if and only if the kernel of the natural epimorphism $\check M \to \check N$
of the corresponding mild $\check A$-promodules is mild.
\end{lem}

\begin{proof}
First suppose that $N$ is a normal quotient
and write $N = M \sslash P$. 
Take a basis $(M_{d})$ of open $A$-submodules of $M$. Then the kernel of 
$\check M \to \check N$ is $(P/(P\cap M_{d}))$, which is clearly mild. 

Conversely suppose that the kernel $K$ of $\check M \to \check N$ is
mild. Then we may suppose that $\check M =(M_{d})_{d\in D}$,
$K=(K_{d})_{d\in D}$   and for each $d \in D$, $K_{d} \subseteq M_{d}$.
Put $L_{d} := \bigcap_{d' \ge d} \Im (K_{d'} \to K_{d})$
and $L := (L_{d})$. Since $K$ is Mittag-Leffler, the natural morphism
$L \to K$ is an isomorphism, and $L$ is equal to $K$
as a subobject of $\check M$.
So 
\[
\check N = \check M / K = \check M /L = (M_{d}/L_{d}).
\]
Now we easily see that $N = M \sslash \hat L$ and so $N$ is a normal
quotient.
\end{proof}

If $A$ is an admissible ring, $(I_{d})$ is a basis of ideals of definition
and $J \subseteq A$ is a closed ideal, then $B:=(A/(J+I_{d}))$
is an admissible proring.
So
the normal quotient $A \sslash J = \check B$ is an admissible ring.

\begin{defn}
Let $X$ be a formal scheme. 
A formal subscheme $Y \subseteq X $ 
is said to be {\em strict}  if
for every $y \in Y$, there exists an affine neighborhood $U \subseteq X$
such that $Y \subseteq U$ is closed and
$\hat \cO_{Y} (U)$ is a normal quotient of $\hat \cO_{X} (U)$.
\end{defn}

If $A$ is a Noetherian adic ring, then every ideal $I$ of $A$
is closed (see \cite[page 264]{Zariski-Samuel-II} or \cite[Th.\ 8.2 and 8.14]{Matsumura}).
Then $A/I=A\sslash I$ and $\Spf A/I$ is
a strict closed formal subscheme of $\Spf A$. 
Thus for a locally Noetherian formal scheme $\cX$,
the formal closed subscheme of $\cX$ is  the same as
the closed subscheme in \cite{EGA}.

\begin{prop}
Let $X$ be a mild formal scheme and $Y $ a closed formal subscheme.
Then $Y$ is strict if and only if the defining proideal sheaf of $Y$ is
mild.
\end{prop}

\begin{proof}
The proposition follows from Lemma \ref{lem-normal-Mittag}.
\end{proof}

\begin{prop}
Every strict formal subscheme of a gentle formal scheme is gentle.
\end{prop}

\begin{proof}
Obvious.
\end{proof}

\begin{cor}
Every strict closed formal subscheme of a gentle affine formal scheme $\Spf A$
is  $\Spf A / I$ for some closed ideal $I \subseteq A$. 
\end{cor}

\begin{proof}
Let $Y \subseteq X:= \Spf A$ be a closed formal strict subscheme
and $\cI \subseteq \cO_{X}$ the defining proideal sheaf, which
is gentle and Mittag-Leffler. 
Now the completion $I$ of $\cI(X)$ is a closed ideal of $A$
and $Y = \Spf A \sslash I$.
\end{proof}

\begin{prop}\label{prop-strict-sub}
Let $A$ be an admissible ring, $J \subseteq A$ a closed ideal and $B:=A \sslash J$.
If $A$ is adic (resp.\ Noetherian and adic), then so is $B$.
Correspondingly every closed formal strict subscheme of an adic (resp.\ locally
Noetherian) formal scheme is adic (resp.\ locally Noetherian).
\end{prop}

\begin{proof}
In both cases, since $A$ is gentle, $B=A/J$. So if $A$ is Noetherian, so is $B$.
It remains to show that if $A$ is adic, so is $B$. 
If $A$ has the $I$-adic topology for some ideal $I \subseteq A$,
then $B=\varprojlim A/(J +I^{n}) =\varprojlim B/ \bar I ^{n} $, where $\bar I :=IB$.
It shows that $B$ has the $\bar I$-adic topology. We have proved the assertion.
\end{proof}

\begin{prop}
Every (non-formal) subscheme of a mild formal scheme is strict.
\end{prop}

\begin{proof}
It is enough to consider a closed subscheme $Y$ of an affine 
formal scheme $X:=\Spf A$.
Then $Y$ is also affine, say $Y = \Spec B$. 
If we write $A =(A_{d})$, then every map $A_{d} \to B$ representing
$A \to B$ is surjective. So $\hat A \to B$ is surjective. Its kernel $I \subseteq \hat A$
is an open ideal. So the quotient topology on $\hat A/I$ is discrete,
so $B = \hat A/I = \hat A \sslash I$. Thus $Y$ is strict.
\end{proof}

\begin{prop}
Let $X$ be a mild formal scheme, $Y \subseteq X$
a strict formal subscheme and $\phi:W \to X$ a morphism of mild formal schemes.
Then the formal subscheme $Y \times _{X} W \subseteq W$ is strict.
\end{prop} 

\begin{proof}
We may suppose that $Y \subseteq X$ is closed.
If $\cI_{Y} \subseteq \cO_{X}$  denotes the defining proideal sheaf,
then we have the exact sequence
\[
0 \to \cI_{Y} \to \cO_{X} \to \cO_{W} \to 0.
\]
Since $\phi^{*}$ is right exact, the sequence
\[
\phi^{*}\cI_{Y} \to \cO_{W} \to \cO_{Y \times_{X}W} \to 0
\]
is exact. Hence the defining proideal of $Y \times_{X}W \subseteq W$
is the image of $\phi^{*}\cI_{Y} \to \cO_{W} $.
It is mild because $\phi^{*}\cI_{Y}$ is so, which proves the proposition.
\end{proof}

\begin{prop}\label{prop-surjective-adic}
Let $A$ be a Noetherian adic ring and
$(I_{d})_{d \in D}$
a directed set of ideals of definition in $A$ (not necessarily a basis of
ideals of definition),
and let $B := \varprojlim A/I_d$. Suppose that $B$ is adic.
Then the natural map $A \to B$ is surjective. Moreover for every $d$, 
the topology on $B$
is the $(I_{d}^n B)_{n \in \NN}$-topology.
\end{prop}

\begin{proof}
Being pro-Noetherian and adic, from Proposition \ref{prop-adic+pro-Noetherian-Noetherian},
 $B$ is Noetherian.
For each $d$, put
\[
\hat I _d := \lim _{\substack{ \longleftarrow \\ d' \ge d }} I_d /I_{d'} .
\]
Fix $e \in D$, and set $J := \hat I_e $.
Since $J$ is an ideal of definition, again from Proposition \ref{prop-adic+pro-Noetherian-Noetherian}, 
the topology on $B$ is identical to the $J$-adic topology. 
Since $(\hat I_d)$ is also a basis of ideals of definition of $B$, 
for every  $n$, there exists $d \in D$ such that $\hat I_d \subset J^{n} $.
Then we have
\[
 B /J^n =( B / \hat I_d ) /  ( J^{n} /   \hat I_d) = (A /I_d)/(I_e/I_d)^{n}.
\]
Hence there exists $d \in D$ such that the kernel of $B/J^m \to B/J^n$ is 
\[
 (I_e /I_d)^{n}/(I_e/I_d)^{m} = I_e^{n} (B/J^m)  .
\]
 Besides $B/J = A /I_e$ is clearly  a finitely generated $A$-module.
As a result, the projective systems  $(A/I_e^{i})$ and $(B/ J^n)$ satisfy the conditions of 
\cite[0, Prop.\ 7.2.9]{EGA}, and hence  
\[
J^n = \Ker (B \to B / J^n) = I_e^n B.
\]
This shows the second assertion.

For each $d$, the map $ A/I_d  \to B /I_d B$ is surjective
and $  B $ is separated for the $ ( I_d^n B)$-topology.
From \cite[Th.\ 8.4]{Matsumura}, $A \to B$ is surjective.
\end{proof}

\begin{cor}\label{cor-strict-adic}
Let $X$ be a locally Noetherian formal scheme and $Y \subseteq X$ a
closed formal subscheme with $Y_{\red} = X_{\red}$. 
Then $Y$ is strict if and only if $Y$ is adic.
\end{cor}

\begin{proof}
It is a consequence of Propositions \ref{prop-strict-sub} and \ref{prop-surjective-adic}.
\end{proof}

\begin{expl}\label{ex-kxy-nonadic2}
Let $A =k[[x,y]]$ be the admissible ring as in Example \ref{ex-kxy-nonadic}
and $A^{\adic}$ be the same ring $k[[x,y]]$ endowed with the 
$(xy)$-adic topology. 
The identity map $A^{\adic} \to A$ is a morphism of admissible rings.
Then $A$ is not a normal quotient. 
 So $\Spf A$ is a closed formal subscheme
of $\Spf A^{\adic}$ which is not strict.
\end{expl}

\begin{expl}\label{expl-infinite-embedded-points}
Suppose that the ring $\CC[x][[t]]$ is endowed with the $(t)$-adic topology.
Put $X := \Spf \CC[x][[t]]$. 
The underlying topological space of $\cX$ is identified with that of $\AA^1= \Spec \CC[x]$.
For each $a \in \CC$,
we define 
\[
Y_a := \Spec \CC[x][[t]]/  (t^2,(x-a)t),
\]
which is a closed subscheme of $X$ and
 has an embedded point at $a \in \CC=\AA^1(\CC)$. 
For a finite subset $S \subseteq \CC$, 
we define  $Y_{S}$
to be the subscheme of $X$ that is
isomorphic to $Y_{a}$ around each $a \in S$
and to $\AA^1$ outside $S$.

Let $T$ be a subset of $\CC$.
Then the $Y_{S}$ with finite $S \subseteq T$ 
form an inductive system.
Define a closed formal subscheme  $Y$ of $X$ by
\[
Y :=\varinjlim_{\text{finite }S \subseteq T} Y_{S}.
\]
Then 
\begin{align*}
\hat \cO_{Y,p} \cong 
\begin{cases}
( k[x,y]/(y^2,xy))_{(x,y)} & (p \in  T) \\
k[x]_{(x)} & (p \in \CC \setminus T) \\
k(x) & (p \text{ the generic point}).
\end{cases}
\end{align*}
Thus all complete stalks of $\hat \cO_Y$ are discrete. 
If $Y$ is locally Noetherian, then it is impossible that infinitely many complete stalks of $\cO_{Y}$
have an embedded prime.
Therefore if $T$ is an infinite set, then  $Y$ is not Noetherian nor
a strict formal subscheme. 
Moreover for every open subscheme $U \subseteq X$, 
$Y \cap U$ is not a strict closed formal subscheme of $U$ either.
If $T$ is uncountable, then $Y$ is not gentle (but mild).
\end{expl}

\begin{thm}\label{thm-single-point}
Let $X$ be a locally Noetherian formal scheme. 
Every closed formal subscheme of $\cX$ is strict
if and only if the underlying topological space of $X$ is discrete.
\end{thm} 

\begin{proof}
The ``if'' direction is essentially due to Chevalley \cite[Lem.\ 7]{Chevalley} (see also \cite[Ch.\ VIII, \S 5, Th.\ 13]{Zariski-Samuel-II}).
To show this, we may suppose that the underlying topological space of $X$ consists of a single point.
Then for some Noetherian complete local ring  $(A,\fm)$ with the $\fm$-adic topology,
we have $X \cong \Spf A$.
There exists a directed set $(I_{d})$
of open ideals of $A$ such that
\[
 Y = \varinjlim _{d} \Spec A /I_d.
\] 
Replacing $A$ with $A / \bigcap  I_{d} $,
we may suppose that $\bigcap  I_{d}  = 0$.
Then for each $n \in \NN$, since $A/\fm^{n}$ is  Artinian,
there exists $d $ such that
$I_{d} (A/\fm^{n})=0$, equivalently $I_{d } \subseteq \fm^{n}$.
Conversely  for every $d $,
there exists $n \in \NN$ with $\fm^{n} \subseteq I_{d}$.
Thus the $(I_d)$-topology coincides with the $\fm$-adic topology, and so
$Y = X$.

We now prove the ``only if'' direction. 
Suppose that the underlying topological space of $X$ is not discrete.
Then there exists a closed but not open point $x$ of $X$.
Let $\Spf A \subseteq X$ be an affine neighborhood of $x$.
Then $\Spf A$ has at least two points. 
Let $A_\red$ be the reduced ring associated to $A$, 
that is, the ring $A$ modulo the ideal of nilpotent elements.
Then $\Spf A$ and $\Spf A_\red$ have the same underlying topological space. 
If $\hat A_\red $ is the $\fm $-adic completion of $A_\red$ with $\fm$ the maximal ideal
of $x$, then $\Spf \hat A_\red$ is a closed formal subscheme of $\Spf A _\red$ consisting of a single point,
hence not isomorphic to $\Spf A_\red$.
Being injective, the natural map $A _\red \to \hat A_\red$ does not factors as $A_\red \to A_\red/ J \cong \hat A_\red $
 for any nonzero ideal $J$. Hence $\Spf \hat A_\red$ is not a closed formal strict subscheme of 
 either $\Spf A_\red$
 or of $\Spf A$.
\end{proof}

As a consequence of a theorem in \cite{Heinzer-Rotthaus}, 
Bill Heinzer showed the following (see the first page of \cite{AJL-correction}): 

\begin{thm}\label{thm-Heinzer}
 Let $k$ be a field. There exists a nonzero ideal $I \subseteq k[x^\pm,y,z][[t]]$ 
 with $ I \cap  k[x,y,z][[t]]= (0)$.
\end{thm}

Suppose that $ k[x,y,z][[t]]$ and $k[x^\pm,y,z][[t]]$ are given the 
$(t)$-adic topologies. Then $X := \Spf k[x^\pm,y,z][[t]]$ is an open 
formal subscheme of $\tilde X :=  k[x,y,z][[t]]$.
Let $I \subseteq k[x^\pm,y,z][[t]]$ be an ideal as in the theorem
and $Y := \Spf k[x^\pm,y,z][[t]]/I$. The theorem says that
there is no closed strict formal subscheme $\tilde Y \subseteq \tilde X$
with $Y=\tilde Y \cap X$. So,
The closure $\bar Y$ of $Y$ in $\tilde X$ is not strict.

There exists also a simpler example:

\begin{thm}\label{thm-example-no-closure}
Consider an element  of  $\CC [x^{\pm },y][[t]]$
\[
 f:= y + a_1 x^{-1}t + a_2 x^{-2}t^2 + a_3 x^{-3}t^3 +\cdots , \ a_i \in \CC \setminus \{0\}.
\]
Suppose that the function $ i\mapsto |a_i| $ is strictly increasing and 
\[
\lim_{i \to \infty} \frac{ |a _{i+1}|}{|a_i|} = \infty.
\] 
Then 
\[
(f) \cap \CC[x,y][[t]] = (0) .
\]
\end{thm}

\begin{proof}
We prove the assertion by contradiction. So we suppose that
there exists $0 \ne g = \sum _{i \in \No} g_i t^i \in \CC  [x^{\pm },y][[t]]$ with $g_i \in \CC[x^\pm,y]$
such that  $ h := fg \in \CC[x,y][[t]]$. 
If we write $ h = \sum_{i \in \No} h_i t^i $ with $h_i \in \CC[x,y]$,
then for every $i \in \No$, we have
\begin{align*}\label{eq-h-g}
 h_ i  = y g_i + \sum_{j=1}^{i} a_j x ^{- j } g_ {i-j}  = y g_i + h_i' ,\ (h_i'   := \sum_{j=1}^{i} a_j x ^{- j } g_ {i-j}) .
\end{align*}
In what follows, we will show that
for sufficiently large $i$, the bottom term of 
$h_{i+1}'$ (that is, the lowest term in the lexicographic order)
is lower in $y$-order than that of $h_{i}'$, which leads to a contradiction.

For each $i \in \No$, write
\[
g_{i}= \sum _{m \in \ZZ, n \in \No} g_{imn} x^{m} y^{n} ,\ g_{imn} \in \CC.
\]
We set
\begin{align*}
 d_i & := \inf \{ m \in \ZZ | \exists n , g_{imn} \ne 0 \} \ \text{(the order of $g_{i}$ in $x$)},\\
 e_i & := \inf \{ n \in \No | g_{id_in} \ne 0  \} \ \text{(the order of $ \sum _{n}g_{id_{i}n}$ in $y$)},  \\
 D_i & := \inf \{ d _ {i-j} -j | 1 \le j \le i   \} =\inf \{ d_{j'} -i +j' | 0 \le j' \le i-1 \} \\
 & \text{(the infimum of the orders of  $x^{-j} g_{i-j}$, $1 \le j \le i$)}, \\
 E_i & := \inf \{ e_{ i-j} | 1 \le j \le i, \ d_{i-j}-j=D_i  \} \\
 &= \inf \{ e_{j'}| 0 \le j' \le i-1, \ 
 d_{j'}-i +j' = D_i \} \\
 & \text{(the infimum of the $y$-orders of those terms } \\
 &\text{in the $x^{-j}g_{i-j}$, $1 \le j \le i$ which are of $x$-order $D_{i}$)}.
\end{align*}
Here by convention, $\inf \emptyset = + \infty$.
We easily see that for every $i' > i$, 
\[
 D_{i'} < D_i \text{ and } E_{i'} \le E_i.
\]

If for $i_0 \in \NN$, $D_{i_0} < 0$ and 
if  the coefficient of $ x^{D_{i_0}} y^{E_{i_0}} $ in $h'_{i_0}$
is nonzero, then  the coefficient of $ x^{D_{i_0}} y^{E_{i_0}} $ in $yg_{i_0}$
is also nonzero.
Moreover if either ``$m < D_{i_0}$'' or ``$m = D_{i_0}$ and $n < E_{i_0} $'', then
 the coefficient of $ x^m y^n $ in $yg_{i_0}$ vanishes.
It follows that 
\[
  d_{i_0 } = D_{i_0} \text{ and } e_{ i_0} = E_{i_0},
\]
and that
\[
  D_{i_0 +1} = D_{i_0} -1 \text{ and } E_{ i_0+ 1} = E_{i_0}-1,
\]
and that the coefficient of  $ x^{D_{i_0+1}} y^{E_{i_0+1}} $ in $h'_{i_0+1}$
is again nonzero. As a result, $E_{i+1} = E_i -1$ for every $i \ge i_0$.
Since  $E_i \in \No$ for every $i$, it is impossible.

Now it remains to show that for some $i \in \NN$,  $D_i <0$ and 
the coefficient of $x^{D_i}y^{E_i}$ in $h'_i$ is nonzero.
Suppose by contrary that for every $i \in \NN$ with $D_i <0$, 
the coefficient of $x^{D_i}y^{E_i}$ in $h'_i$ is zero. 
Since $i \mapsto D_i$ is strictly decreasing, there exists 
 $ i_1 \in \NN $ such that for every $i \ge i_1$, $D_i < 0$.
Then  for every $i \ge i_1$, the coefficient of $x^{D_i} y ^{E_i}$ in $yg_i $
must be zero.
Therefore we have
\[
 D_ i = D_{i_1} - ( i - i_1 ) \text{ and } E_i = E_{i_1}.
\]
Let  
\[
\Lambda := \{ j| \text{the coefficient of $x^{D_{i_1}} y^{E_{i_1}}$ in $ x^{-j} g_{i_1-j} $ is nonzero} \}
 \subseteq \{ 1,2 ,\dots, i_1 \}
\] 
and let $0 \ne c_j \in \CC $ be the coefficient of $x^{D_{i_1}} y^{E_{i_1}}$ 
in $x^{-j} g_{i_1-j}$, $j \in \Lambda$.
For every $i \ge i_1$, the  coefficient of $x^{D_i}y^{E_i}$ in $h_i'$ is
\[
   \sum _{j \in \Lambda}a_{ j +i-i_1} c_j  = 0.
\]
Let $j_0 \in \Lambda$ be the largest element and $j_1 \in \Lambda$ the second largest one. 
(Note that $\sharp \Lambda \ge 2$).
From the assumption on the  $a_i$, for $i \gg i_1$,
we have
\[
| a_{ j_0 +i-i_1} | - 
( \sharp \Lambda -1)|  a_{ j_1 +i-i_1} | (\max _{j \in \Lambda \setminus \{j_0\} } |c_j /c_{j_0}| )>0.
\]
Therefore, for $i \gg 0$,
\begin{align*}
0&= |\sum _{j \in \Lambda}a_{ j +i-i_1} c_j| \\
& \ge  |c_{j_0}|  \left( | a_{ j_0 +i-i_1} | - \sum_{j \in \Lambda \setminus \{j_0\} }  
|  a_{ j +i-i_1}  c_j  /c_{j_0}  |  \right)\\
& \ge  |c_{j_0}|  \left( | a_{ j_0 +i-i_1} | - 
( \sharp \Lambda -1)|  a_{ j_1 +i-i_1} | (\max _{j \in \Lambda \setminus \{j_0\} } |c_j /c_{j_0}| )\right)\\
& >0 
\end{align*}
This is a contradiction. 
We have proved  the theorem.
\end{proof}

If we remove one more variable, then there is no ideal as in Theorems \ref{thm-Heinzer} and \ref{thm-example-no-closure}:

\begin{prop}
Let $k$ be a field.
Then for  any nonzero ideal $I$ of  $ k[x^\pm][[t]]$, 
$ I \cap k[x][[t]] \neq (0)$.
\end{prop}

\begin{proof}
It  suffices to prove the assertion in
 the case where $I$ is principal, say  $ I= (f)$, $f \in k[x^\pm][[t]] $.
Write
\[
f = \sum_{ i \ge n} f_i t^i \in k[x^\pm][[t]], \ f_i \in k[x^\pm], \ f_n \ne 0.
\]
Define  $g_i \in k[x ^\pm, f_n^{-1}]$ inductively as follows;
\[
 g_0 := f_n^{-1} , g _ {i+1} := - ( \sum _{ 0 \le j \le i} g_j f_{n+ i +1 -j} )/ f_n .
\] 
Then 
\begin{align*}
 f   (\sum_{i \ge 0} g_i t^i) &=  \sum _{m \ge n} (( f_n g_{m-n} + \sum _{\substack{i+j =m \\ j < m-n}}f_ig_j ) t^m) \\
 &= t^n + \sum _{m > n} (( - \sum _{j< m-n}g_jf_{m-j} + \sum _{\substack{i+j =m \\ j < m-n}}f_ig_j ) t^m) \\
& = t^n .
\end{align*}
Since $\sum_{i \ge 0} g_i t^i$ is invertible,
ideals $ (f)  $ and $(t^n)$ of $ k[ x^\pm, f_n^{-1}][[t]] $ are identical.
Now if necessary, replacing $f$ with $x^{i}f$ for some  $i \in \ZZ$,
we may suppose that $f_{n} \in k[x] \setminus (x)$. 
Then $\Spf k[x][[t]] $ is covered by the open formal subschemes
$ \Spf k[x^{\pm}][[t]] $ and  $\Spf k[x,f_{n}^{-1}][[t]]$.
Gluing $ \Spf k[x^\pm][[t]]/(f) $ and $\Spf k[x, f_n^{-1}][[t]] /(t^n)$, we obtain 
a closed formal strict subscheme $Z$ of $ \Spf k[x][[t]] $.
Since $ Z$ contains $ \Spf k[x^\pm][[t]]/(f)$ as an open formal subscheme, 
$Z$ is not identical to $\Spf k[x][[t]]$.
Hence
  $Z$ is defined by a nonzero ideal $J \subseteq k[x][[t]]$.
Therefore 
\[
I \cap  k[x][[t]] \supseteq J \neq (0).  
\]
\end{proof}

\subsection{Locally pre-Noetherian formal schemes and plain formal subschemes}

\begin{defn}
A mild formal scheme $X$ is said to be {\em locally pre-Noetherian} 
if for every $x \in X$, $\hat \cO_{X,x}$ is Noetherian (not necessarily adic).
\end{defn}

\begin{defn}
A formal subscheme $Y$ of a mild formal scheme $X$ is said to be
{\em plain} if for every $y \in Y$, the map of complete stalks $\hat \cO_{X,y} \to \hat \cO_{Y,y}$
is surjective.
\end{defn}

\begin{prop}\label{prop-plain-pre-pre}
A plain formal subscheme of a locally pre-Noetherian formal scheme is again 
locally pre-Noetherian.
\end{prop}

\begin{proof}
Obvious.
\end{proof}

\begin{prop}\label{prop-strict-gentle->plain}
A strict formal subscheme of a gentle formal scheme is plain.
\end{prop}

\begin{proof}
Let $X$ be a gentle formal scheme and $Y$
be a gentle and strict formal subscheme. 
If necessary, replacing $X$ with an open formal subscheme, we may
suppose that $Y$ is a closed formal subscheme, say with the defining 
proideal $\cI \subseteq \cO_{X}$. 
Then for each $y \in Y$, we have the short exact sequence of
mild $\cO_{X,y}$-modules
\[
0 \to \cI_{y} \to \cO_{X,y} \to \cO_{Y,y} \to 0 .
\]
Hence we have $\hat \cO_{Y,y} = \hat \cO_{X,y} \sslash \hat \cI_{y}$.
Since $\hat \cO_{X,y}$ is gentle, $\hat \cO_{X,y} \sslash \hat \cI_{y}= \hat \cO_{X,x} / \hat \cI_{y}$.
As a consequence, the natural map $\hat \cO_{X,y} \to \hat \cO_{Y,y}$
is surjective and $Y$ is a plain formal subscheme.
\end{proof}

\begin{prop}
Every (non-formal) subscheme of a mild formal scheme is plain.
\end{prop}

\begin{proof}
Let $Y$ be a subscheme of a mild formal scheme $X$.
For each $y \in Y$, $\cO_{Y,y}$ is a ring and $\cO_{X,y}$ is an admissible ring.
If we write $\cO_{X,y} =(A_{d})$, then every morphism $A_{d} \to \cO_{Y,y}$
representing $\cO_{X,y} \to \cO_{Y,y}$ is surjective.
Since $\cO_{X,y}$ is mild, $\hat \cO_{X,y} \to A_{d} \to \cO_{Y,y} = \hat \cO_{Y,y}$
is surjective, which proves the proposition.
\end{proof}

\begin{expl}
Let $A $ and $A^{\adic}$ be as in Example \ref{ex-kxy-nonadic} 
and \ref{ex-kxy-nonadic2}. 
The formal schemes $ X:= \Spf A $ and $X ^\adic := \Spf A^\adic $
 have the same underlying topological space,
which consists of three open prime ideals, $(x,y)$,  $ ( x)  $ and $(y)$. 
The complete stalks of $\cO_X$ and $\cO _{X^\adic}$ at $ (x,y) $ and $(y)$ are identical as rings,
but not at $ (x) $. 
We have 
\[
\hat \cO_{X, (x)} = k((y))[[x]] / (x) =k((y)) \text{ and } \hat  \cO_{X^\adic, (x)} = k((y))[[x]].
\] 
It follows that via the morphism $X \to X^\adic$
induced by the identity map $A^\adic \to A$, $X$ is a plain formal subscheme of $X^{\adic}$
but not strict.
\end{expl}

\begin{expl}
With the notation as in Example \ref{expl-infinite-embedded-points}, if
$T \subseteq \CC$ is infinite, then 
$Y \subseteq X$ is plain but not strict. 
\end{expl}

\section{Formal separatrices of singular foliations}\label{sec-foliation}

In this section, we construct non-adic formal schemes
from singularities of foliations. 

\subsection{Formal separatrices}

Let $X$ be a smooth algebraic variety over $\CC$, and $ \Omega _{X}=\Omega_{X/\CC}$ 
the sheaf of (algebraic) K\"ahler differential forms.
A {\em (one-codimensional) foliation} on $X$ is an invertible saturated subsheaf
$\cF$ of $\Omega _{X}$ satisfying the integrability condition: $ \cF \wedge d\cF =0 $.
We say that a foliation $\cF$ is {\em smooth} at $x \in X$ if
the quotient sheaf $\Omega _X /\cF$ is locally free around $x$, and
that $\cF$ is {\em singular} at $x$ otherwise.
We say that $\cF$ is {\em smooth} if $\cF$ is smooth at every point.
The pair $(X,\cF)$ of a smooth variety  $X$ and a foliation on $X$
is called a {\em foliated variety}.
 
\begin{defn}
Let $(X,\cF)$ be a foliated variety, $x \in X (\CC)$, $ X_{/x} := \Spf \hat \cO_{X,x}$,
 $Y \subseteq  X_{/x}$  a strict closed formal 
 subscheme of codimension one defined by $0 \neq f \in \hat \cO_{X,x}$,
and $ \omega \in \Omega_{X,x} $ a generator of $\cF_x$.
We say that $Y$ is  a {\em formal separatrix} (of $\cF$) at $x$ if
$f$ divides $ \omega \wedge d f $. 
\end{defn}

Frobenius theorem says that if $\cF$ is smooth at $x$,
there exists a unique smooth formal separatrix of $\cF$ at $x$.
Miyaoka \cite{Miyaoka} proved  that  the family of smooth formal separatrices
at smooth points of a foliation is a formal scheme:

\begin{thm}\label{thm-Miyaoka}\cite[Cor.\ 6.4]{Miyaoka}
Let $(X,\cF)$ be a foliated variety. Suppose that $\cF$ is smooth.
Then there exists a strict closed formal 
subscheme $L$ of  $(X \times _\CC X)_{/ \Delta_X}$ such that
for every point $x \in X$, $ p_2( p_1^{-1}(x)) $ is the smooth formal separatrix of $\cF$ at $x$.
Here $\Delta_X \subseteq X \times _\CC X$ is the diagonal, 
$(X \times _\CC X)_{/ \Delta_X}$ is the completion of $X \times _\CC X$ along $\Delta_X$
and $p_{1},p_{2}:L \to X$ 
are the first and second projections.
\end{thm}

Let $(X,\cF)$ be a foliated variety and  $C \subseteq X$ a closed smooth subvariety of dimension 1.
Suppose that $C$ meets only at a single point $ o $ with the singular
locus of $\cF$. Let $U \subseteq X$ be the smooth locus of $\cF$
and $L \subseteq (U \times _\CC U)_{/\Delta _U}$ the family of formal separatrices
as in the theorem.
Then $C \setminus \{o\}$ is a closed subvariety of $U$.
The fiber product 
\[
L_{C\setminus \{o\}} := ( C \setminus \{o\}) \times _{U, p_1} L
\]
 is the family of
the smooth formal separatrices over $C \setminus \{o\}$,
and a strict formal subscheme of $Z_{C}:=(C \times _\CC X)_{/\Delta_ C}$.
Let $L _C := \overline{ L_{ C\setminus \{o\} } }$ be the closure of
$ L_{ C\setminus \{o\} }$ in $Z_{C}$.

\begin{prop}\label{prop-equivalent-conditions-closed}
The following are equivalent:
\begin{enumerate}
\item ${L_{C}}$ is locally Noetherian.
\item ${L_{C}}$ is adic.
\item $ {L_{C}}$ is locally pre-Noetherian.
\item $ {L_{C}} \subseteq Z _{C}$ is strict.
\item ${L_{C}} \subseteq Z_{C}$ is plain.
\end{enumerate}
\end{prop}

\begin{proof}
$1 \Rightarrow  2 $  and $1 \Rightarrow 3$: Trivial.

$2 \Rightarrow 1$:  Corollary \ref{cor-ind-adic-Noet}. 

$3 \Rightarrow 2$: 
Let $p \in L_{C}$ be the point over $o \in C$,
$A :=\hat \cO_{L_{C},p}$ and $ I \subseteq A $ the largest ideal of definition,
which is prime.
From the construction of $L_{C}$, the symbolic powers $ I^{(n)} $
form a basis of ideals of definition in $A$. 
It is easy to see that $A$ is a domain. 
If $\fm \subseteq A$ is the maximal ideal of $p$,
then  the $\fm$-adic completion of $A$ is also a domain.
From \cite[page 33, Lem.\ 3]{Zariski} (see also \cite[Ch.\ VIII, \S 5, Cor.\ 5]{Zariski-Samuel-II}), 
the topology on $A$ is equal to the $I$-adic topology. 

Now let $\Spf B \subseteq L_{C}$ be an affine open with $B$ an admissible ring
and $J \subseteq B$ the largest ideal of definition.
Then we see that the topology on $B$ is the $J$-adic topology,
which prove the assertion.

$2 \Rightarrow 4 $:  Corollary \ref{cor-strict-adic}.

$4 \Rightarrow 5$: Proposition \ref{prop-strict-gentle->plain}.

$5 \Rightarrow 3$: Proposition \ref{prop-plain-pre-pre}.
\end{proof}

\begin{thm}\label{thm-limit-formal-separatrix}
Suppose that one of the conditions in Proposition \ref{prop-equivalent-conditions-closed} holds.
Then the fiber of $  {L_C} \to C $ over $o$ is
a formal separatrix at $o$. 
\end{thm}

\begin{proof}
We need to use complete modules of differentials of locally 
Noetherian formal schemes. 
For a morphism $f:W \to V$ of locally Noetherian formal schemes,
we have a complete module of differentials, $\hat \Omega_{W/V}$, 
which is a semicoherent complete $\hat \cO_W$-module, and a derivation 
$\hat d_{W/V}: \hat \cO_W \to \hat \Omega_{W/V}$.
We refer to \cite{AJP} for details.

If necessary, shrinking $X$, we can take a nowhere vanishing 
$\omega \in \cF (X)$.
Let 
\[
 \psi: Z_{C} \to X.
\]
be the projection.
Pulling back $\omega$, we obtain a global section $\psi^* \omega$ of $ \hat \Omega_{Z_{C}/C} $.
Since $L_C$ is a hypersurface in $Z_{C}$, it is defined by a section $f$ of $\hat \cO_{Z_{C}}$. 
Since the restriction of $L_C$ to $C \setminus \{o\}$ is
 the family of formal separatrices along $C \setminus \{o\}$, $f$
 divides $\psi^* \omega \wedge \hat d_{Z_{C}/C}f $. 

Let $ Y$ be the fiber of $L _C \to C$ over $o$, which is a hypersurface of $X_{/o}$
defined by the image $\bar f \in \hat \cO_{X,o}$ of $f$.
Then $\bar f$ divides $\omega \wedge \hat d_{X_{/o}/\CC} \bar f$. 
Hence $Y$ is a formal separatrix. 
\end{proof}

\subsection{Jouanolou's theorem}

We recall Jouanolou's result on Pfaff forms.
We refer to  \cite{Jouanolou} for details.

An {\em  algebraic Pfaff form of degree $m$}  on $\PP_\CC ^2$ is a one-form
\[
 \omega = \omega _1 dx + \omega_3 dy + \omega_3 dz
\]
such that $\omega _i$ are homogeneous polynomials of degree $m$ and the equation
\[
  x \omega _1 + y \omega _2 + z \omega _3 =0
\] 
holds. 
A {\em Pfaff equation of degree $m$} on $\PP_\CC^2$ is a class of algebraic Pfaff forms 
modulo nonzero scalar multiplications. 

Let $\omega$ be an algebraic Pfaff form on $\PP_\CC^2$ and $ [\omega] $ its Pfaff equation class.
An {\em algebraic solution} of  $\omega$ or $[\omega]$ is a class of  homogeneous polynomials 
$f \in \CC[x,y,z]$
modulo nonzero scalar multiplications
such that $f$ divides $\omega \wedge df $.

Let $V_m$ be the vector space of the algebraic Pfaff forms of degree $m$ on $\PP_\CC^2$.
Then the set of the Pfaff equations of degree $m$ on $\PP_\CC^2$
is identified with the projective space $\PP (V_m)= (V_m \setminus \{0\} )/ \CC^*$.
Define 
\[
 Z_m \subseteq \PP (V_m)
\]
to be the set of the Pfaff equations that have no algebraic solution.

\begin{thm}\cite[\S 4]{Jouanolou}
Suppose $m \ge 3$. 
Then $Z_m$ is the intersection of countably many nonempty Zariski open subsets of $\PP(V_m)$
and contains the class of the algebraic Pfaff form
\[
 (x^{m-1}z -y^m ) dx + (y^{m-1}x -z^m) dy +(z^{m-1}y -x^m ) dz .
\]
\end{thm}

From \cite[page 4, Prop.\ 1.4]{Jouanolou}, every algebraic Pfaff form $\omega$ on $\PP_\CC^2$ is
integrable: $  d \omega \wedge \omega = 0$. So $\omega$ defines also a foliation $\cF_\omega$ on $\CC^3$.
From \cite[page 85, Prop.\ 2.1]{Jouanolou}, the only singular point of $\cF _\omega$
is the origin. Accordingly we can define 
the family $L_{\omega, C \setminus \{o\}}$ of formal separatrices along $C \setminus \{o\}$
and its closure
$L _{\omega,C}$ for any line $C \subset \CC^3$ through the origin.

Let $ f = \sum_{i \ge n} f_i  \in \CC [[x,y,z]] $. Here $f_i$ is a homogeneous polynomial of degree $i$
and $f_n \ne 0$. Suppose that $f$ defines a formal separatrix at the origin, equivalently
that $f$ divides $\omega \wedge df$.
Then the class of $f_n$ is an algebraic solution of the Pfaff equation $[\omega]$.
Hence if $[\omega ] \in Z_m$, then $\cF_\omega$ has no formal separatrix at the origin.

\begin{cor}\label{cor-last-corollary}
For $[\omega] \in Z_m$ and a line $C \subseteq \CC^3$ through the origin, 
the formal subscheme $L _{\omega , C}$ of $ Z_{C}\cong \Spf \CC[w][[x,y,z]]$
is neither strict, plain, locally pre-Noetherian nor adic.
\end{cor}

\begin{proof}
If $L_{\omega,C}$ is either strict, plain, locally pre-Noetherian or adic,
then from Theorem \ref{thm-limit-formal-separatrix},
the foliation $\cF_\omega$ has a formal separatrix at the origin.
Hence $[\omega] \notin Z_m$, a contradiction.
\end{proof}


\end{document}